\documentclass[bj,numbers,noshowframe]{imsart}

\usepackage[hang]{footmisc}

\usepackage{verbatim}
\newcommand{\comments}[1]{}

\usepackage{enumitem}

\usepackage{amsmath,amsbsy,amsgen,amscd,amsthm,amsfonts,amssymb} 

\usepackage{mathtools}
\mathtoolsset{centercolon}  

\startlocaldefs
\numberwithin{equation}{section}

\theoremstyle{plain}
\newtheorem{theorem}{Theorem}[section]
\newtheorem{lemma}[theorem]{Lemma}
\newtheorem{proposition}[theorem]{Proposition}
\newtheorem{corollary}[theorem]{Corollary}

\theoremstyle{remark}
\newtheorem{assumption}[theorem]{Assumption}

\newtheorem{remark}[theorem]{Remark}

\newcommand{\insquare}[1]{\left[#1\right]}             

\newcommand{\mcA}{\mathcal{A}}

\newcommand{\upr}{u}
\newcommand{\upost}{\upsilon}

\usepackage[usenames,dvipsnames]{xcolor}

\let\originalleft\left
\let\originalright\right
\renewcommand{\left}{\mathopen{}\mathclose\bgroup\originalleft}
\renewcommand{\right}{\aftergroup\egroup\originalright}

\renewcommand{\phi}{\varphi}
\newcommand{\eps}{\varepsilon}

\providecommand{\mathbbm}{\mathbb} 

\newcommand{\R}{\mathbbm{R}}
\newcommand{\E}{\mathbbm{E}}

\newcommand{\kk}{k}     
\newcommand{\CC}{\mathcal{C}}     

\newcommand{\A}{\mathcal{A}}

\newcommand{\mcN}{\mathcal{N}}

\newcommand{\msK}{\mathscr{K}}

\definecolor{mygreen}{rgb}{0.13,0.55,0.13}

\newcommand{\iid}{\stackrel{\text{i.i.d.}}{\sim}}

\newcommand{\Nc}{\mathcal{N}}

\newcommand{\mcF}{\mathcal{F}}

\usepackage{float}

\endlocaldefs

\begin{document}

\begin{frontmatter}
\title{Covariance Operator Estimation: Sparsity, Lengthscale, and Ensemble Kalman Filters}
\runtitle{Covariance Operator Estimation via Thresholding}

\begin{aug}
\author{\inits{}\fnms{O.}~\snm{Al-Ghattas}\ead[label=]{}}
\author{\inits{}\fnms{J.}~\snm{Chen}\ead[label=]{}}
\author{\inits{}\fnms{D.}~\snm{Sanz-Alonso}\ead[label=]{}}
\author{\inits{}\fnms{N.}~\snm{Waniorek}\ead[label=]{}}
\address{University of Chicago}

\end{aug}

\begin{abstract}
This paper investigates covariance operator estimation via thresholding. For Gaussian random fields with approximately sparse covariance operators, we establish non-asymptotic bounds on the estimation error in terms of the sparsity level of the covariance and the expected supremum of the field. We prove that thresholded estimators enjoy an exponential improvement in sample complexity compared with the standard sample covariance estimator if the field has a small correlation lengthscale. As an application of the theory, we study thresholded estimation of covariance operators within ensemble Kalman filters.
\end{abstract}


\end{frontmatter}


\section{Introduction}\label{sec:introduction}
This paper studies thresholded estimation of the covariance operator of a Gaussian random field. Under a sparsity assumption on the covariance model, we bound the estimation error in terms of the sparsity level and the expected supremum of the field. Using this bound, we then analyze covariance operator estimation in the interesting regime where the correlation lengthscale is small, and show that the thresholded covariance estimator achieves an exponential improvement in sample complexity compared with the standard sample covariance estimator. As an application of the theory, we demonstrate the advantage of using thresholded covariance estimators within ensemble Kalman filters. 

The first contribution of this paper is to lift the theory of covariance estimation from finite to infinite dimension.
In the finite-dimensional setting, a rich body of work \cite{wu2003nonparametric, bickel2008regularized,el2008operator,cai2012adaptive,cai2012minimax,cai2012optimal,chen2012masked,wainwright2019high,ghattas2022non} shows that, exploiting various forms of sparsity, it is possible to consistently estimate the covariance  matrix  
of a vector $u \in \R^{d_u}$ with $N \sim \log (d_u)$ samples. The sparsity of the covariance  matrix
---along with the use of thresholded, tapered, or banded estimators that exploit this structure--- facilitates an 
exponential improvement in sample complexity relative to the unstructured case, where $N \sim d_u$ samples are needed \cite{bai2008limit, gordon1985some, vershynin2010introduction}. 
In this work we investigate the setting 
in which $u$ is an infinite-dimensional random field with an approximately sparse covariance model. Specifically, 
 we generalize notions of approximate sparsity often employed in the finite-dimensional covariance estimation literature \cite{bickel2008covariance, cai2012optimal}.
We show that the statistical error of thresholded estimators can be bounded in terms of the expected supremum of the field and the sparsity level, the latter of which quantifies the rate of spatial decay of correlations of the random field.
Our analysis not only lifts existing theory from finite to infinite dimension, but also provides non-asymptotic  
moment bounds not yet available in finite dimension.

The second contribution of this paper is to showcase the benefit of thresholding in the challenging regime where the correlation lengthscale of the field is small relative to the size of the physical domain. While a vast literature in nonparametric statistics \cite{ghosal2017fundamentals} and approximation theory \cite{wendland2004scattered} highlights the key role of smoothness in determining optimal convergence rates for many nonparametric estimation tasks, our non-asymptotic theory emphasizes that the lengthscale rather than the smoothness of the covariance function drives the difficulty of the estimation problem and the advantage of thresholded estimators. Fields with small correlation lengthscale are ubiquitous in applications. For instance, they arise naturally in climate science and numerical weather forecasting, where global forecasts need to account for the effect of local processes with a small correlation lengthscale, such as cloud formation or propagation of gravitational waves. We show that thresholded estimators achieve an exponential improvement in
sample complexity: For a field with lengthscale $\lambda$ in $d$-dimensional physical space, the standard sample covariance requires $N \sim \lambda^{-d}$ samples, while thresholded estimators only require $N \sim \log (\lambda^{-d})$. Therefore, our theory suggests that the parameter $\lambda^{-d}$ plays the same role in infinite dimension as $d_u$ in the classical finite-dimensional setting.
To analyze thresholded estimators in the small lengthscale regime, we use our general non-asymptotic moment bounds and the sharp scaling of sparsity level and expected supremum with lengthscale.


The third contribution of this paper is to demonstrate the advantage of using thresholded covariance estimators within ensemble Kalman filters \cite{evensen2009data}. Our interest in covariance operator estimation was motivated by the widespread use of localization techniques within ensemble Kalman methods in inverse problems and data assimilation, see e.g. \cite{houtekamer2001sequential, houtekamer2016review, farchi2019efficiency,tong2022localized,chen2021auto}. Many inverse problems in medical imaging and the geophysical sciences are most naturally formulated in function space \cite{stuart2010inverse,bui2013computational,bigoni2020data}; likewise, data assimilation is primarily concerned with sequential estimation of spatial fields, e.g. temperature or precipitation \cite{kalnay2003atmospheric,carrassi2018data}. Theoretical insight for these applications calls for sparse covariance estimation theory in function space, which has not been the focus in the literature. Perhaps partly for this reason, the empirical success of localization techniques in ensemble Kalman methods is poorly understood, with few exceptions that study localization in finite dimension \cite{tong2018performance,ghattas2022non}. The work \cite{sanz2023analysis} studies the behavior of ensemble Kalman methods under mesh discretization, but it does not consider localization. In this paper, we use our novel non-asymptotic covariance estimation theory to obtain a sufficient sample size to approximate an idealized mean-field ensemble Kalman filter using a localized ensemble Kalman update. In finite dimension, \cite{ghattas2022non} studies the ensemble approximation of mean-field algorithms for inverse problems and \cite{ghattas2023ensemble} conducts a multi-step analysis of ensemble Kalman filters without localization.

The paper is organized as follows. We first state and discuss our three main theorems in the following section. Then, the next three sections contain the proof of these theorems, along with further auxiliary results of independent interest. We close with conclusions, discussion, and future directions.

{\bf{Notation}}
Given two positive sequences $\{a_n\}$ and $\{b_n\}$, the relation $a_n\lesssim b_n$ denotes that $a_n \le c b_n$ for some constant $c>0$. If the constant $c$ depends on some quantity $\tau$, then we write $a \lesssim_{\tau} b$. If both $a_n \lesssim b_n$ and $b_n \lesssim a_n$ hold simultaneously, then we write $a_n \asymp b_n$. For a finite-dimensional vector $a$, $|a|$ denotes its Euclidean norm. For an operator $\A$, $\|\A\|$ denotes its operator norm, $\A^{*}$ its adjoint, and $\text{Tr}(\A)$ its trace.

\section{Main Results}\label{sec:mainresults}
This section states and discusses the main results of the paper. In Subsection \ref{ssec:MainResult1} we analyze
the thresholded sample covariance estimator 
in a general setting, and establish moment bounds in Theorem \ref{thm:mainresult}. In Subsection \ref{ssec:MainResult2}
we consider a small lengthscale regime, and show in Theorem \ref{thm:smalllengthscale} that the thresholded estimator significantly improves upon the standard sample covariance estimator. 
Finally, in Subsection \ref{ssec:MainResult3} we apply our new covariance estimation theory to demonstrate the advantage of using thresholded covariance estimators within ensemble Kalman filters. 

\subsection{Thresholded Estimation of Covariance Operators}\label{ssec:MainResult1}
Let $u, u_1,u_2, \ldots,u_N $ be \textit{i.i.d.} centered almost surely continuous Gaussian random functions on $D=[0,1]^d$  taking values in $\R$  with covariance function (kernel) $\kk: D \times D \to \R$ and covariance operator $\CC: L^2(D) \to L^2(D),$ so that, for $x,x' \in D$ and $\psi \in L^2(D),$
\[
\kk(x,x'):=\E \Bigl[u(x)u(x')\Bigr], 
\qquad 
(\CC \psi)(\cdot) := \int_D \kk(\cdot,x') \psi (x') \, dx'
.
\]
The sample covariance function $\widehat{\kk}(x,x')$ and sample covariance operator $\widehat{\CC}$ are defined analogously by
\[
\widehat{\kk}(x,x') :=\frac{1}{N}\sum_{n=1}^N u_n(x)u_n(x'),
\qquad 
(\widehat{\CC}\, \psi)(\cdot) := \int_D \widehat{\kk}(\cdot,x')\psi(x')\, dx'
.
\]
We introduce the thresholded sample covariance estimators with thresholding parameter $\rho_N$
\[
\widehat{\kk}_{\rho_N}(x,x'):=\widehat{\kk}(x,x')\mathbf{1}_{\{|\widehat{\kk}(x,x')|\ge \rho_N\}}(x,x'),
\qquad 
(\widehat{\CC}_{\rho_N} \,\psi)(\cdot) := \int_D \widehat{\kk}_{\rho_N}(\cdot,x') \psi(x') \, dx'
,
\]
where $\mathbf{1}_A$ denotes the indicator function of the set $A$.
Our first main result, Theorem \ref{thm:mainresult} below, relies on the following general assumption: 

\begin{assumption}\label{assumption:main1}
  $u, u_1,u_2, \ldots,u_N $ are \textit{i.i.d.} centered almost surely continuous Gaussian random functions on $D=[0,1]^d$ taking values in $\R$ with covariance function $k$. Moreover, the following holds: 
    \begin{enumerate}[label=(\roman*)]
        \item $\sup_{x\in D} \E \big[u(x)^2\big]=1$. \label{assumption:magnitude}
        \item  For some $q \in (0,1)$ and $R_q>0$, $\sup_{x\in D}\left(\int_{D} |\kk(x,x')|^q \,dx'\right)^{\frac{1}{q}} \leq R_q$. \qed \label{assumption:sparsity} 
    \end{enumerate}
\end{assumption}

We assume fully observed functional data and defer extensions to partially observed data \cite{james2000principal,james2003clustering,yao2005functional,yao2005functionalbis,qiao2020doubly,fang2023adaptive} to future work.
Assumption \ref{assumption:main1} (i) normalizes the fields to have unit maximum marginal variance over $D.$  Assumption \ref{assumption:main1} (ii) generalizes standard notions of sparsity in finite dimension to our infinite-dimensional setting  ---refer e.g. to \cite{bickel2008covariance,cai2012optimal,wainwright2019high}, which study estimation of a covariance matrix $\Sigma = (\sigma_{ij}) \in \mathbb{R}^{d_u \times d_u}$ under the row-wise approximate sparsity assumption that $\max_i \sum_{j=1}^{d_u} |\sigma_{ij}|^q \le \widetilde{R}_q^q.$


  Our first main result establishes moment bounds on the 
deviation of the thresholded covariance estimator from its target in terms of the approximate sparsity level $R_q$ and the expected supremum of the field, the latter of which determines the scaling of $\rho_N$. We prove Theorem \ref{thm:mainresult} and several auxiliary results of independent interest in Section \ref{sec:thresholdedestimation}. 

\begin{theorem}\label{thm:mainresult}
Suppose that Assumption \ref{assumption:main1} holds. Let $1\le c_0\le \sqrt{N}$ and set 
\begin{align}
    \rho_N &:= c_0\left[\frac{1}{N}\lor \frac{1}{\sqrt{N}}\E\Big[\sup_{x\in D} u(x)\Big]\lor \frac{1}{N}\Big(\E\Big[\sup_{x\in D} u(x)\Big]\Big)^2\right], \label{eq:thresholdingestimator} \\
    \widehat{\rho}_N &:= c_0\left[ \frac{1}{N}\lor \frac{1}{\sqrt{N}}\Big(\frac{1}{N}\sum_{n=1}^{N}\sup_{x\in D} u_{n}(x)\Big)\lor \frac{1}{N}\Big(\frac{1}{N}\sum_{n=1}^{N}\sup_{x\in D} u_{n}(x)\Big)^2\right]. \label{eq:thresholdparameter}
\end{align}
Then, for any $p\ge 1$,
\begin{equation}\label{eq:in-expectation-bound}
\big[\E \|\widehat{\CC}_{\widehat{\rho}_N}-\CC\|^p\big]^{\frac{1}{p}} \lesssim_p R_q^q\rho_{N}^{1-q}+\rho_Ne^{-\frac{c}{p} N\big(\,\rho_N\land\, \rho_N^2\big)},
\end{equation}
where $c$ is a universal constant.
\end{theorem}

An appealing feature of Theorem \ref{thm:mainresult} is that it holds for \emph{any} sample size $N \ge 1.$ The following immediate corollary provides a simplified statement which holds for sufficiently large sample size.

 \begin{corollary}\label{corollary:mainresult}
Suppose that Assumption \ref{assumption:main1} holds and that $\sqrt{N} \ge \E\Big[\sup_{x\in D} u(x)\Big] \ge \frac{1}{\sqrt{N}}.$ Set
\begin{align*}
    \rho_N :=\frac{1}{\sqrt{N}}\E\Big[\sup_{x\in D} u(x)\Big],  \qquad 
    \widehat{\rho}_N := \frac{1}{\sqrt{N}}\Big(\frac{1}{N}\sum_{n=1}^{N}\sup_{x\in D} u_{n}(x)\Big). 
\end{align*} 
Then, for any $p\ge 1$,
\begin{equation*}
\big[\E \|\widehat{\CC}_{\widehat{\rho}_N}-\CC\|^p\big]^{\frac{1}{p}} \lesssim_p R_q^q\rho_{N}^{1-q}+\rho_Ne^{-\frac{c}{p} N \rho_N^2},
\end{equation*}
where $c$ is a universal constant.
\end{corollary}

 To the best of our knowledge, Theorem \ref{thm:mainresult} and Corollary \ref{corollary:mainresult} are the first results in the literature to consider covariance operator estimation under the natural sparsity Assumption~\ref{assumption:main1} (ii). 
  As will be discussed next, the first of the two terms in the right-hand side of \eqref{eq:in-expectation-bound} is reminiscent of existing results for covariance matrix estimation. The second term in \eqref{eq:in-expectation-bound} depends only on the expected supremum of the field, and, as we will show in Subsection \ref{ssec:MainResult2}, it is negligible in the small lengthscale regime.
  
For covariance matrix estimation under $\ell_q$-sparsity, \cite[Theorem 6.27]{wainwright2019high} proves that if the sample covariance matrix satisfies $|\widehat{\Sigma}_{ij} - \Sigma_{ij}| \lesssim \widetilde{\rho}_N$ for all $1 \le i,j\le d_u,$ then the error of an estimator with thresholding parameter $\widetilde{\rho}_N$ can be bounded by $\widetilde{R}_q^q \widetilde{\rho}_N^{1-q},$ where $\widetilde{R}_q$ is a quantity analogous to our $R_q$ that controls the row-wise $\ell_q$-sparsity of $\Sigma$. This explains the choice of thresholding parameter 
$$\widetilde{\rho}_N \asymp \frac{1}{\sqrt{N}}  \sqrt{ \log d_u} \asymp \frac{1}{\sqrt{N}} \mathbb{E} \insquare{ \max_{1 \le i \le d_u}u_i }$$
in finite dimension, which ensures an entry-wise control on the sample covariance matrix with high probability. Analogously, our infinite-dimensional theory relies on sup-norm bounds for the sample covariance function $\widehat{k}(x,x');$  we obtain these bounds in Subsection \ref{ssec:covariancefunction} using tools from  empirical process theory. For instance, Proposition \ref{prop:single_supremum} shows that with our choice of thresholding parameter $\rho_N$, we have $\sup_{x \in D}| \widehat{k}(x,x') - k(x,x')| \lesssim \rho_N$ with high probability.   
Therefore, Theorem \ref{thm:mainresult} and Corollary \ref{corollary:mainresult} reveal that the expected supremum is the key dimension-free quantity that determines the choice of thresholding parameter and the error of estimation in both finite and infinite-dimensional settings.  Since in practice the expected supremum of the field (and hence $\rho_N$) is unknown, we replace it with $\widehat{\rho}_N$ to define a computable thresholded estimator $\widehat{\mathcal{C}}_{\widehat{\rho}_N}$. The concentration of $\widehat{\rho}_N$  around $\rho_N$ is established in Lemma \ref{lemma:hat_rhoN_rhoN}. 



\begin{remark}
  In contrast to existing results in the finite-dimensional setting (see e.g. \cite{bickel2008covariance, cai2012optimal,wainwright2019high}) that provide in-probability bounds or moment bounds of order up to $p=2$, Theorem \ref{thm:mainresult} provides moment bounds for all $p \ge 1$. For example,  \cite[Theorem 6.27]{wainwright2019high} shows a high-probability statement where $\widetilde{\rho}_N$ necessarily depends on the desired confidence level. Consequently, \cite[Theorem 6.27]{wainwright2019high} cannot be used to derive moment bounds of arbitrary order. In contrast, Theorem~\ref{thm:mainresult} shows that the tuning parameter of the covariance operator estimator need not be tied to the confidence level. The proof technique therefore contributes to the literature on confidence parameter independent estimators; see e.g. \cite{bellec2018slope} for an analogous finding that, contrary to standard practice \cite{bickel2009simultaneous}, the Lasso tuning parameter need not depend on the confidence level.
\end{remark}

\begin{remark}
    The proof of the small lengthscale results in Subsections \ref{ssec:MainResult2} and \ref{ssec:MainResult3} utilizes Theorem \ref{thm:mainresult} with a careful choice of thresholding parameter prefactor $c_0$. However, the exponential improvement in sample complexity established in Theorems \ref{thm:smalllengthscale} and \ref{thm:LEKIExpectation} holds for any fixed value $c_0\gtrsim 1.$
    As noted in \cite[Section 3]{bickel2008covariance} and \cite[Section 4]{cai2011adaptive}, establishing an optimal choice of prefactor $c_0$ is challenging even in the simpler setting of covariance matrix estimation, where $c_0$ is often taken as a fixed constant or chosen empirically through cross-validation \cite{bickel2008covariance,cai2011adaptive,cai2012adaptive}.
   We will numerically showcase in Subsection \ref{ssec:MainResult2} the exponential improvement of a thresholded estimator with the choice $c_0=5.$    \qed
\end{remark}

\begin{remark}
As in the finite-dimensional setting \cite{cai2012optimal, el2008operator}, our thresholded estimator $\widehat{\CC}_{\widehat{\rho}_N}$ is positive semi-definite with high probability, but it is not guaranteed to be positive semi-definite. Fortunately, a simple modification ensures positive semi-definiteness while maintaining the same order of estimation error achieved by the original estimator. Notice that $\widehat{\CC}_{\widehat{\rho}_N}$ is a self-adjoint and Hilbert-Schmidt operator since $\int_{D\times D} \big|\widehat{\kk}_{\rho_N}(x,x')\big|^2 dxdx'<\infty$, see \cite[Example 9.23]{hunter2001applied}. Therefore, there is an orthonormal basis $\{\phi_i\}_{i=1}^{\infty}$ of $L^2(D)$ consisting of eigenfunctions of $\widehat{\CC}_{\widehat{\rho}_N}$ such that $\widehat{\kk}_{\rho_N}(x,x')=\sum_{i=1}^{\infty}\widehat{\lambda}_i \phi_i(x)\phi_i(x'),$  where $\widehat{\lambda}_i$ is the $i$-th eigenvalue of $\widehat{\CC}_{\widehat{\rho}_N}$. Let $\widehat{\lambda}_i^{+}=\widehat{\lambda}_i \lor 0$ be the positive part of $\widehat{\lambda}_i$ and define
\[
\widehat{\kk}_{\rho_N}^{+}(x,x'):=\sum_{i=1}^{\infty}\widehat{\lambda}_i^{+} \phi_i(x)\phi_i(x'),
\qquad 
(\widehat{\CC}_{\rho_N}^{+} \,\psi)(\cdot) := \int_D \widehat{\kk}_{\rho_N}^{+}(\cdot,x') \psi(x') \, dx' .
\]
Then, $\widehat{\CC}_{\rho_N}^{+}$ is positive semi-definite and further
\begin{align*}
\|\widehat{\CC}_{\rho_N}^{+}-\CC\|&\le \|\widehat{\CC}_{\rho_N}^{+}-\widehat{\CC}_{\rho_N}\|+\|\widehat{\CC}_{\rho_N}-\CC\|\le \max_{i:\widehat{\lambda}_i\le 0} |\widehat{\lambda}_i|+ \|\widehat{\CC}_{\rho_N}-\CC\|\\
&\le\max_{i:\widehat{\lambda}_i\le 0} |\widehat{\lambda}_i-\lambda_i|+ \|\widehat{\CC}_{\rho_N}-\CC\|\le 2\|\widehat{\CC}_{\rho_N}-\CC\|,
\end{align*}
where $\lambda_i$ is the $i$-th eigenvalue of $\CC$. Thus, $\widehat{\CC}_{\rho_N}^{+}$ is positive semi-definite and attains the same estimation error as the original thresholded estimator $\widehat{\CC}_{\rho_N}$. In light of this fact, we will henceforth assume that $\widehat{\CC}_{\rho_N}$ is positive semi-definite wherever needed. \qed
\end{remark}

\subsection{Small Lengthscale Regime}\label{ssec:MainResult2}

Our second main result, Theorem \ref{thm:smalllengthscale},  shows that in the small lengthscale regime thresholded estimators enjoy an exponential improvement in sample complexity relative to the sample covariance estimator. To formalize this regime, we introduce the following additional assumption:
\begin{assumption}\label{assumption:main2}
    The following holds: 
        \begin{enumerate}[label=(\roman*)]
        \item 
        The covariance function $k$ is isotropic and positive, so that $k(x,x')=k(|x-x'|)>0$. Moreover, $k(r)$ is differentiable, strictly decreasing on 
        $[0, \infty),$ and satisfies $k(r)\to 0$ as $r\to\infty.$ \label{assumption:fieldandkernel}
        \item  The kernel $k=k_{\lambda}$ depends on a correlation lengthscale parameter $\lambda>0$ such that $k_{\lambda}(\alpha r)=k_{\lambda\alpha^{-1}}(r)$ for any $\alpha>0$, and $k_{\lambda}(0)=k(0)$ is independent of $\lambda.$  \qed \label{assumption:smalllengthscalelimit}
    \end{enumerate}
\end{assumption}

Assumption \ref{assumption:main2} (i) requires isotropy of the covariance kernel on $D;$ this assumption, while restrictive, is often invoked in applications \cite{williams2006gaussian,stein2012interpolation}.  Assumption \ref{assumption:main2} (ii) makes explicit the dependence of the kernel on the correlation lengthscale parameter $\lambda.$ As discussed later,
the nonparametric Assumption \ref{assumption:main2} is satisfied by important parametric covariance functions, such as squared exponential and Matérn models. 
 The \textit{small lengthscale regime} holds whenever the underlying covariance function satisfies Assumption~\ref{assumption:main2} and $\lambda$ is sufficiently small. In the scientific applications that motivate our work, the dimension of the physical space is small ($d= 1,2,3$). Hence, we will treat $d$ as a constant in our analysis of the small lenghtscale regime. Theorem \ref{thm:smalllengthscale} compares the errors of sample and thresholded covariance estimators.
The proof can be found in Section \ref{sec:smalllengthscale}.

\begin{theorem}\label{thm:smalllengthscale}
Suppose that Assumptions \ref{assumption:main1} and \ref{assumption:main2} hold. Let $c_0\gtrsim 1$ be an absolute constant and set
\[
\widehat{\rho}_N
:=
\frac{c_0}{\sqrt{N}}\Big(\frac{1}{N}\sum_{n=1}^{N}\sup_{x\in D} u_{n}(x)\Big).
\]
There is a universal constant $\lambda_0>0$ such that for $\lambda< \lambda_0$ and $N\gtrsim \log (\lambda^{-d})$, the sample covariance estimator and the thresholded covariance estimator satisfy
\begin{align}
    \frac{\mathbb{E} \| \widehat{\CC} - \CC \|}{\|\CC\|} & \asymp \bigg(\sqrt{\frac{\lambda^{-d}}{N}} \lor \frac{\lambda^{-d}}{N}\bigg), \label{eq:smalllengthscalewithoutthresholding} \\
    \frac{\mathbb{E} \| \widehat{\CC}_{\widehat{\rho}_N} - \CC \|}{\|\CC\|} &\le \,c(q)\bigg(\frac{\log(\lambda^{-d})}{N}\bigg)^{\frac{1-q}{2}},
\label{eq:smalllengthscalewiththresholding}
\end{align}
where $c(q)$ is a constant that depends only on $q$.
\end{theorem}
\begin{remark} The term $c(q)$ in \eqref{eq:smalllengthscalewiththresholding} admits a form
    \[c(q) \asymp \frac{\int_0^{\infty} k_1(r)^q r^{d-1}dr}{\int_0^{\infty} k_1(r)r^{d-1}dr},
    \]
    where $k_1(r)$ is the kernel function with correlation lengthscale parameter $\lambda=1$ in Assumption \ref{assumption:main2}. As an explicit example, for the squared exponential kernel defined in~\eqref{eq:SEandMaterncovariance}, we have $k_1(r) = e^{-r^2/2}$ and a straightforward calculation shows that $c(q) \asymp q^{-d/2}$.
    \qed 
\end{remark}

Theorem \ref{thm:smalllengthscale} shows that, for sufficiently small $\lambda,$ we need $N \gtrsim \lambda^{-d}$ samples to control the relative error of the sample covariance estimator, while $N \gtrsim \log (\lambda^{-d})$ samples suffice to control the relative error of the thresholded estimator. The error bound in \eqref{eq:smalllengthscalewiththresholding} is reminiscent of the convergence rate $s_0\big(\frac{\log d_u }{N}\big)^{(1-q)/2}$ of thresholded estimators for $\ell_q$-sparse matrices $\Sigma \in \R^{d_u \times d_u}$   with sparsity level $s_0$  \cite{bickel2008covariance, cai2012optimal}. Therefore, Theorem \ref{thm:smalllengthscale} indicates that, in our infinite-dimensional setting, the parameter $\lambda^{-d}$ plays an analogous role to $d_u$ and  $c(q)$  plays an analogous role to $s_0$. However, we remark that the estimation error in Theorem \ref{thm:smalllengthscale} is \emph{relative error}, whereas in the finite-dimensional covariance matrix estimation literature \cite{bickel2008covariance, cai2012optimal, cai2011adaptive}, the estimation error is often \emph{absolute error}. While in the finite-dimensional setting the sparsity parameter $s_0$ may increase with $d_u,$ the constant $c(q)$ in our bound \eqref{eq:smalllengthscalewiththresholding} is independent of the lengthscale parameter $\lambda$. Moreover, inspired by the minimax optimality of thresholded estimators for $\ell_q$-sparse covariance matrix estimation \cite{cai2012optimal}, we conjecture that the convergence rate \eqref{eq:smalllengthscalewiththresholding} is also minimax optimal, and we intend to investigate this question in future work. 

The bound \eqref{eq:smalllengthscalewithoutthresholding} for the sample covariance estimator relies on the seminal work \cite{koltchinskii2017concentration}, which shows that, for any sample size $N,$
    \begin{align}\label{eq:Koltchinksiibound}
         \frac{\mathbb{E} \| \widehat{\mathcal{C}} - \mathcal{C} \|}{\|\mathcal{C} \| }  \asymp  \sqrt{\frac{r(\mathcal{C})}{N} }   \lor  \frac{r(\mathcal{C})}{N}, 
         \qquad  
         r(\mathcal{C}) := \frac{\text{Tr}(\mathcal{C})}{\| \mathcal{C}\|}.
    \end{align}
Consequently, \eqref{eq:smalllengthscalewithoutthresholding} follows by a sharp characterization of the operator norm and the trace of $\CC$ in terms of $\lambda$. In contrast, the bound \eqref{eq:smalllengthscalewiththresholding} for the thresholded estimator relies on our new Theorem~\ref{thm:mainresult}, and requires an analogous characterization of the thresholding parameter $\rho_N$ and approximate sparsity level $R_q$ in terms of $\lambda$. 

In the remainder of this subsection, we illustrate Theorem \ref{thm:smalllengthscale} with a simple numerical experiment where we consider the estimation of covariance operators for squared exponential  (SE) and Mat\'ern  (Ma)  models in dimensions $d=1$ and $d=2$ at small lengthscales.
We emphasize that our theory is developed under mild nonparametric assumptions on the covariance kernel
 as outlined in Assumption~\ref{assumption:main2}; 
however, for simplicity here we focus on two important parametric models.  
For $x,x' \in D$, define the corresponding covariance functions
\begin{align}\label{eq:SEandMaterncovariance}
k^{\mathrm{SE}}_\lambda(x, x') &:=\exp \left(-\frac{|x-x'|^2}{2\lambda^2}\right), \\
k^{\mathrm{Ma}}_\lambda(x, x') &:=\frac{2^{1-\nu}}{\Gamma(\nu)}\left(\frac{\sqrt{2\nu}}{\lambda}|x-x'|\right)^\nu K_\nu\left(\frac{\sqrt{2\nu}}{\lambda}|x-x'|\right),\label{eq:Matern}
\end{align}
where $\Gamma$ denotes the Gamma function and $K_\nu$ denotes the modified Bessel function of the second kind. In both cases, the parameter $\lambda$ is interpreted as the correlation lengthscale of the field and Assumption \ref{assumption:main2}  is satisfied. Moreover, Assumption \ref{assumption:main1} is satisfied by the squared exponential model, and it is satisfied by the Mat\'ern model provided that the smoothness parameter $\nu$ satisfies $\nu>(\frac{d-1}{2}\lor \frac{1}{2})$. We refer to \cite[Lemma 4.2]{sanz2022unlabeled} for the almost sure continuity of random samples and to \cite[Appendix 3, Lemma 11]{nobile2015multi} for the H\"older continuity of the Mat\'ern covariance function $k^{\mathrm{Ma}}(r)$. For the Mat\'ern model, we take the smoothness parameter to be $\nu=3/2$ in our experiments.

We will report results in physical dimension $d=1$ and $d=2.$ To respectively resolve small lengthscales up to order $\lambda \asymp 10^{-3}$ and $\lambda \asymp 10^{-2},$ we discretize the domain $D=[0,1]$ with a mesh of $L=1250$ uniformly spaced points and the domain $D=[0,1]^2$ with $L = 10,000$ points.
In the $d=1$ case we consider a total of $30$ lengthscales arranged uniformly in log-space and ranging from $10^{-3}$ to $10^{-0.1}$, and in the $d=2$ case we consider a total of $10$ lengthscales arranged  in log-space and ranging from $10^{-2.3}$ to $10^{-0.1}$.
For each lengthscale $\lambda$, with corresponding covariance operator $\CC$, the discretized covariance operators are given by the $L\times L$ covariance matrices
\begin{align*}
\CC^{ij}:=k(x_i,x_j), 
\qquad 
1\leq i,j\leq L,
\end{align*}
and we sample $N = 5\log(\lambda^{-1})$ realizations of a Gaussian process on the mesh, denoted $u_1,\ldots u_N \sim \Nc(0,\CC)$. We then compute the empirical and thresholded sample covariance matrices 
\begin{align*}
    \widehat{\CC}^{ij}:=\frac{1}{N}\sum_{n=1}^N u_{n}(x_i)u_{n}(x_j),  \qquad \widehat{\CC}_{\widehat{\rho}_N}^{\,ij}:=\widehat{\CC}^{ij}\mathbf{1}_{\{|\widehat{\CC}^{ij}|\geq \widehat{\rho}_N\}}, 
    \qquad 1\leq i,j\leq L,
\end{align*}
scaling the thresholding level $\widehat{\rho}_N$ as described in Theorem \ref{thm:mainresult}.

\begin{figure}[h]
\centering
\includegraphics[scale=.369]{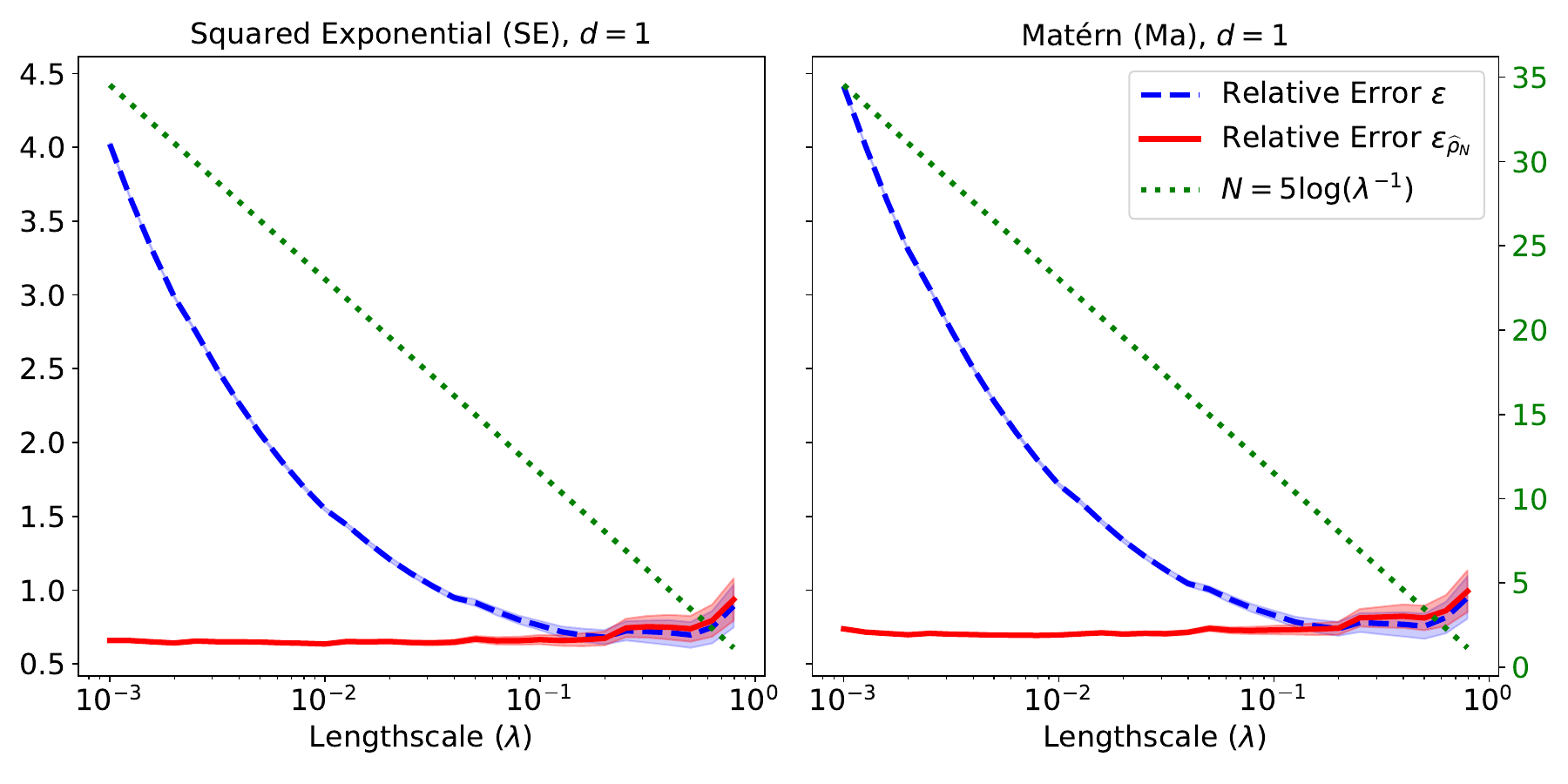}
\caption{Plots of the average relative errors and 95\% confidence intervals achieved by the sample ($\varepsilon$, dashed blue) and thresholded ($\varepsilon_{\widehat{\rho}_N}$, solid red) covariance estimators based on sample size ($N$, dotted green) for the squared exponential kernel (left) and Mat\'ern kernel (right) in $d=1$ over 100 trials.}
\label{fig:RelativeErrors}
\end{figure}

To quantify the performance of each of the estimators, we compute their relative errors 
\begin{align*}
     \eps:=\frac{\|\widehat{\CC}-\CC\|}{\|\CC\|}, 
     \qquad 
     \eps_{\widehat{\rho}_N}:=\frac{\|\widehat{\CC}_{\widehat{\rho}_N}-\CC\|}{\|\CC\|}.
 \end{align*}

The experiment is repeated a total of 100 times for each lengthscale in the case $d=1$ and 30 times for each lengthscale in the case $d=2$. In Figure~\ref{fig:RelativeErrors}, we plot average relative errors as well as 95\% confidence intervals over the 100 trials for both squared exponential and Mat\'ern models in $d=1$, along with the sample size for each lengthscale setting. In Figure~\ref{fig:RelativeErrorsd2}, we present the $d=2$ analog of Figure~\ref{fig:RelativeErrors}. Our theoretical results are clearly illustrated: taking only $N =  5 \log(\lambda^{-d})$ samples, the relative error in the thresholded estimator
 remains constant as the lengthscale decreases, whereas the relative error in the sample covariance operator diverges. Notice that Figures~\ref{fig:RelativeErrors} and~\ref{fig:RelativeErrorsd2} also show that thresholding can increase the relative error for fields with large correlation lengthscale. 

\begin{figure}[h]
\centering
\includegraphics[scale=.369]{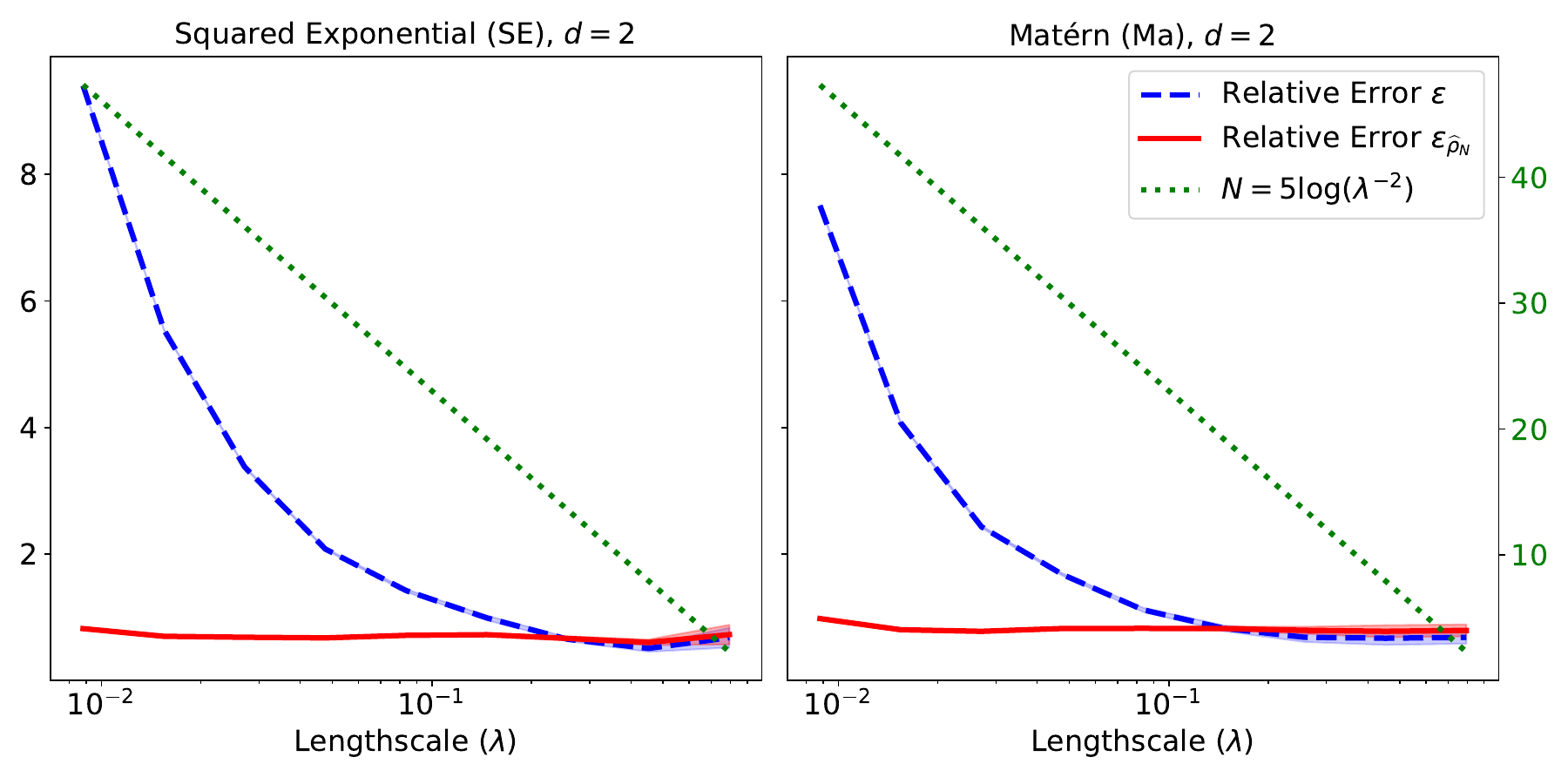}
\caption{Plots of the average relative errors and 95\% confidence intervals achieved by the sample ($\varepsilon$, dashed blue) and thresholded ($\varepsilon_{\widehat{\rho}_N}$, solid red) covariance estimators based on sample size ($N$, dotted green) for the squared exponential kernel (left) and Matérn kernel (right) in $d=2$ over 30 trials.
}
\label{fig:RelativeErrorsd2}
\end{figure}

\subsection{Application in Ensemble Kalman Filters}\label{ssec:MainResult3}
Nonlinear filtering is concerned with online estimation of the state of a dynamical system from partial and noisy observations. Filtering algorithms blend the dynamics and observations by sequentially solving inverse problems of the form
\begin{equation}\label{eq:IP}
y = \mcA u + \eta,
\end{equation}
 where $y \in \R^{d_y}$ denotes the observation, $u \in L^2(D)$ denotes the state, $\mcA: L^2(D) \to \R^{d_y}$ is a linear observation operator, and $\eta \sim  \mcN(0, \Gamma)$ is the observation error with positive definite covariance matrix $\Gamma$. In Bayesian filtering \cite{sanzstuarttaeb}, the model dynamics define a prior or \emph{forecast} distribution on the state, which is combined with the data likelihood implied by the observation model \eqref{eq:IP} to obtain a posterior or \emph{analysis} distribution. In most applications, the update from forecast to analysis distribution must be implemented through an approximate filtering algorithm. For instance, in operational numerical weather forecasting where the state may represent a temperature field along the surface of the Earth, discretizations of size $10^9$ are routinely used to capture small lengthscales on the order of kilometers. In this setting, computing exactly the Kalman formulas that define the forecast-to-analysis update would be unfeasible.      
 
Ensemble Kalman filters (EnKFs) are a rich family of algorithms scalable to highly complex data assimilation tasks \cite{evensen2009data}, including operational numerical weather forecasting \cite{houtekamer2016review}. The key idea behind these methods is to represent forecast and analysis distributions using an ensemble of $N$ particles, so that the computational cost is controlled by the number of particles, which is typically small, rather than by the level of discretization.  For instance, in operational weather forecasting $N \sim 10^2 \ll 10^9;$ we refer to \cite{tippett2003ensemble} for a summary of the computational and memory costs of different EnKFs in terms of the discretization level and the number of particles. Taking as input a forecast ensemble $\{ u_n \}_{n=1}^N \iid \mcN(0, \CC)$ 
 and observed data $y$ generated according to \eqref{eq:IP}, EnKFs produce an analysis ensemble $\{\upost_n\}_{n=1}^N$. 
Each analysis particle $\upost_n$ is obtained by nudging a forecast particle $u_n$ towards the observed data $y.$ The amount of nudging is controlled by a \emph{Kalman gain} operator to be estimated using the first two moments of the forecast ensemble. 
Vanilla implementations of EnKFs rely on the sample covariance, see e.g. \cite[Algorithm 10.2]{sanzstuarttaeb}. However, some form of covariance localization is required for EnKFs to scale to operational settings \cite{houtekamer2001sequential}. While the use of localization within EnKFs is standard, few works have demonstrated its statistical benefit \cite{tong2018performance,ghattas2022non}, and none in the functional setting that is most relevant in applications.  In this subsection we show that thresholded covariance operator estimators within the EnKF analysis step can dramatically reduce the ensemble size required to approximate an idealized, non-implementable, \emph{mean-field} EnKF that uses the population moments of the forecast distribution. Consequently, we identify an ensemble size which suffices for each EnKF particle to be updated similarly as in the limit of infinite number of particles. We refer to \cite{herty2019kinetic,calvello2022ensemble} for recent works that study the behavior of EnKFs in the mean-field limit. 


Define the mean-field EnKF analysis update by 
\begin{align}\label{eq:MFEnKF}
    \upost_n^\star 
    &:=\upr_n+ \msK(\CC ) \bigl(y  - \mcA\upr_n- \eta_n\bigr), \qquad 1\le  n \le N,
\end{align}
where $\{ \eta_n \}_{n=1}^N \iid \mathcal{N}(0, \Gamma)$ and
\begin{align}\label{eq:Kalman_gain_operator}
\msK(\CC):=\CC \A^{*}(\A\CC\A^{*}+\Gamma)^{-1}
\end{align}
denotes the Kalman gain. Practical algorithms do not have access to the forecast distribution, and rely instead on the forecast ensemble to estimate both $\CC$ and $\msK$. 
We will investigate two popular analysis steps, given by
\begin{alignat}{3}
    \upost_n
    &:=\upr_n+ \msK(\widehat{\CC}) \bigl(y  - \mcA\upr_n- \eta_n\bigr), \qquad &&1\le  n \le N, \label{eq:EnKF}\\
    \upost_n^{\,\rho}
    &:=\upr_n+ \msK(\widehat{\CC}_{\rho_N}) \bigl(y  - \mcA\upr_n- \eta_n\bigr),\ \qquad  &&1\le  n \le N. \label{eq:localEnKF}
\end{alignat}
The analysis step in \eqref{eq:EnKF} is known as the \emph{perturbed observation} or \emph{stochastic} EnKF \cite{burgers1998analysis}. For simplicity of exposition, we will assume here that when updating $u_n,$ this particle is not included in the  sample covariance $\widehat{\CC}$ used to define the Kalman gain. This slight modification of the sample covariance will facilitate a cleaner statement and proof of our main result, Theorem \ref{thm:LEKIExpectation}, without altering the qualitative behavior of the algorithm. The analysis step in \eqref{eq:localEnKF} is based on a thresholded covariance operator estimator. Again, we assume that the thresholded estimator $\widehat{\CC}_{\rho_N}$ is defined without using the particle $u_n.$
The following result is a direct consequence of our theory on covariance operator estimation in the small lengthscale regime. The proof can be found in Section \ref{sec:ensembleKalman}.

\begin{theorem}[Approximation of Mean-Field EnKF]\label{thm:LEKIExpectation}
    Suppose that Assumptions \ref{assumption:main1} and \ref{assumption:main2} hold. Let $y$ be generated according to \eqref{eq:IP} with bounded observation operator $\mcA: L^2(D) \to \R^{d_y}$. Let $\upost_n^\star$ be the mean-field EnKF update in \eqref{eq:MFEnKF}, and let $\upost_n$ and $\upost_n^{\,\rho}$ be the EnKF and localized EnKF updates in \eqref{eq:EnKF} and \eqref{eq:localEnKF}. Let $c_0\gtrsim 1$ be an absolute constant and set
    \begin{align*}
        \rho_{N} \asymp \frac{c_0}{\sqrt{N}}\Big(\frac{1}{N}\sum_{n=1}^{N}\sup_{x\in D} u_{n}(x)\Big).
       \end{align*}
   Then, there is a universal constant $\lambda_0>0$ such that for $\lambda< \lambda_0$ and $N\gtrsim \log (\lambda^{-d})$,
    \begin{align*}  
        \E \left[ | \upost_n-\upost_n^{\star} | \mid u_n,\eta_n\right]
        &\lesssim c \bigg(\sqrt{\frac{\lambda^{-d}}{N}} \lor \frac{\lambda^{-d}}{N}\bigg),
         \\
        \E \left[ | \upost_n^{\,\rho}- \upost_n^{\star}| \mid u_n,\eta_n\right]& \lesssim c\Bigg[ c (q)\bigg(\frac{\log(\lambda^{-d})}{N}\bigg)^{\frac{1-q}{2}}\Bigg]
        ,
    \end{align*}
    where $c = \|\A\|\|\Gamma^{-1}\| \|\CC\| |y-\A u_n-\eta_n|$.
\end{theorem}

\section{Thresholded Estimation of Covariance Operators}\label{sec:thresholdedestimation}
This section studies thresholded estimation of covariance operators in the general setting of Assumption \ref{assumption:main1}.
In Subsection \ref{ssec:covariancefunction} we show uniform error bounds on the sample covariance function estimator $\widehat{k}(x,x')$. These results are used in Subsection \ref{ssec:proofmainresult1} to prove our first main result, Theorem \ref{thm:mainresult}.

\subsection{Covariance Function Estimation}\label{ssec:covariancefunction}
In this subsection we establish uniform error bounds on the sample covariance function estimator. These bounds will play a central role in our analysis of thresholded estimation of covariance operators developed in the next subsection. We first establish a high-probability bound, which is uniform over both arguments of the covariance function. 

\begin{proposition}\label{prop:max_norm_bound} 
Under Assumption \ref{assumption:main1}, there exist positive absolute constants $c_1, c_2$ such that, for all $t \geq 1$, it holds with probability at least $1-c_1 e^{-c_2 t}$ that
   \[
    \sup_{x,x'\in D} \left|\widehat{\kk}(x,x')-\kk(x,x')\right|\lesssim \left[\left(\frac{t}{N} \vee \sqrt{\frac{t}{N}}\right)  
    \E\Big[\sup_{x\in D} u(x)\Big] \right] \vee \frac{\left( \E\left[\sup_{x\in D} u(x)\right]\right)^2}{N}.
    \]
\end{proposition}

\begin{proof}
We will apply the product empirical process bound in \cite[Theorem 1.13]{mendelson2016upper}. To that end, define the evaluation functional at $x\in D$ by
    \[
    \ell_x: u \longmapsto \ell_x(u)=u(x)
    \]
and write
    \[
    \left|\widehat{\kk}(x,x')-\kk(x,x')\right|=\left|\frac{1}{N}\sum_{n=1}^{N}u_n(x)u_n(x')-\E\left[u(x)u(x')\right]\right|=\left|\frac{1}{N}\sum_{n=1}^{N}\ell_{x}(u_n)\ell_{x'}(u_n)-\E \left[\ell_x(u)\ell_{x'}(u)\right]\right|,
    \]
so that
    \[
    \sup_{x,x'\in D}\left|\widehat{\kk}(x,x')-\kk(x,x')\right|=\sup_{f,g\in\mathcal{F}}\left|\frac{1}{N}\sum_{n=1}^{N}f(u_n)g(u_n)-\E\left[f(u)g(u)\right]\right|,
    \]
    where $\mathcal{F}:=\{\ell_x\}_{x\in D}$ denotes the family of evaluation functionals. Note that $\{\ell_x\}_{x\in D}$ are continuous linear functionals on $C(D),$ the space of continuous functions on $D$ endowed with its usual topology. We can then apply \cite[Theorem 1.13]{mendelson2016upper} (see also \cite[Theorem B.11]{ghattas2022non}) which implies that, with probability $1-c_1e^{-c_2t}$,
    \begin{gather}\label{eq:C_hat-C}
    \sup_{x,x'\in D}\left|\widehat{\kk}(x,x')-\kk(x,x')\right|\lesssim \left[\left(\frac{t}{N} \vee \sqrt{\frac{t}{N}}\right)\left(\sup _{f \in \mathcal{F}}\|f\|_{\psi_2} \gamma_2\left(\mathcal{F}, \psi_2\right) \right)\right] \vee \frac{ \gamma_2^2\left(\mathcal{F}, \psi_2\right)}{N},
    \end{gather}
    where here and henceforth $\gamma_2$ denotes Talagrand's generic complexity \cite[Definition 2.7.3]{talagrand2022upper} and $\psi_2$ denotes the Orlicz norm with Orlicz function $\psi(x) = e^{x^2}-1,$ see e.g. \cite[Definition 2.5.6]{vershynin2018high}. Since $u$ is Gaussian, the $\psi_2$-norm of linear functionals is equivalent to the $L^2$-norm. Hence,
\begin{gather}\label{eq:aux1}
\sup_{f\in\mathcal{F}}\|f\|_{\psi_2}\lesssim \sup_{f\in\mathcal{F}}\|f\|_{L^2}=\sup_{f\in\mathcal{F}}\sqrt{\E\left[f^2(u)\right]}=\sup_{x\in D}\sqrt{\E\left[ u^2(x)\right]}=\sup_{x\in D}\sqrt{\kk(x,x)} = 1,
\end{gather}
where we used Assumption \ref{assumption:main1} \ref{assumption:magnitude} in the last step.
Next, to control the complexity $\gamma_2\left(\mathcal{F}, \psi_2\right),$ let
\begin{equation*}
    \mathsf{d}(x, x'):=\sqrt{\mathbb{E}\left[(u(x)-u(x'))^2\right]}=\|\ell_x(\cdot)-\ell_{x'}(\cdot)\|_{L^2(P)},\quad x,x'\in D,
\end{equation*}
where $P$ is the distribution of the random function $u$. Then,
\begin{equation}\label{eq:aux2}
  \gamma_2(\mathcal{F},\psi_2)\overset{(\text{i})}{\lesssim}\gamma_2(\mathcal{F},L^2)=\gamma_2(D,\mathsf{d})\overset{(\text{ii})}{\asymp} \E\left[\sup_{x\in D} u(x)\right],  
\end{equation}
where (i) follows by the equivalence of $\psi_2$ and $L^2$ norms for linear functionals and (ii) follows by Talagrand's majorizing-measure theorem \cite[Theorem 2.10.1]{talagrand2022upper}. Combining the inequalities \eqref{eq:C_hat-C}, \eqref{eq:aux1}, and \eqref{eq:aux2} gives the desired result.
\end{proof}


\begin{corollary}\label{cor:max_norm_bound}
Under Assumption \ref{assumption:main1},
     it holds that, for any $p\ge 1$,
    \[
    \left(\E \left[ \sup_{x,x'\in D} \left|\widehat{k}(x,x')-k(x,x')\right|^p\right]\right)^{\frac{1}{p}} \lesssim_p \frac{
    \E\left[\sup_{x\in D} u(x)\right]}{\sqrt{N}}\vee \frac{\left( \E\left[\sup_{x\in D} u(x)\right]\right)^2}{N}.
    \]
\end{corollary}
\begin{proof} The result follows by integrating the tail bound in Proposition \ref{prop:max_norm_bound}. 
\end{proof}

In contrast to Proposition~\ref{prop:max_norm_bound}, the following result provides uniform control over the error when holding fixed one of the two covariance function inputs.
For this easier estimation task, we obtain an improved exponential tail bound that we will use in the proof of Theorem \ref{thm:mainresult}.

\begin{proposition}\label{prop:single_supremum}
Suppose that Assumption \ref{assumption:main1} holds.
 Let $1\le c_0\le N$ and set 
\[
\rho_N := c_0\left[\frac{1}{N}\lor \frac{1}{\sqrt{N}}\E\Big[\sup_{x\in D} u(x)\Big]\lor \frac{1}{N}\Big(\E\Big[\sup_{x\in D} u(x)\Big]\Big)^2\right].
\]
Then, for every $x'\in D$, it holds with probability at least $1-4e^{-c_1 N(\rho_N\land\, \rho_N^2)}$ that 
   \[
    \sup_{x\in D} \,\left|\widehat{\kk}(x,x')-\kk(x,x')\right|\lesssim \rho_N.
    \]
\end{proposition}
\begin{proof}
We will apply the multiplier empirical process bound in \cite[Theorem 4.4]{mendelson2016upper}. To that end, we write 
    \[
    \left|\widehat{\kk}(x,x')-\kk(x,x')\right|=\left|\frac{1}{N}\sum_{n=1}^{N}u_n(x)u_n(x')-\E\left[u(x)u(x')\right]\right|=\left|\frac{1}{N}\sum_{n=1}^{N}\ell_{x}(u_n)\ell_{x'}(u_n)-\E \left[\ell_x(u)\ell_{x'}(u)\right]\right|,
    \]
    so that for the class $\mathcal{F}:=\{\ell_x\}_{x\in D}$ of evaluation functionals and for a fixed $g \in \mcF$, we have
    \begin{align*}
    \sup_{x\in D}\,\left|\widehat{\kk}(x,x')-\kk(x,x')\right|
    &=\sup_{f \in\mathcal{F}}\,\left|\frac{1}{N}\sum_{n=1}^{N}f(u_n)g(u_n)-\E\left[f(u)g(u)\right]\right|\\
    &=\frac{1}{N} \sup_{f \in\mathcal{F}}\,\left|\sum_{n=1}^{N} \big(f(u_n)\xi_n-\E\left[f(u)\xi \right]\big)\right|,
    \end{align*}
    where $\xi_n := g(u_n).$ Note that $\xi_1,\dots, \xi_N$ are i.i.d. copies of $\xi \sim \mathcal{N}\bigl(0,k(x',x')\bigr),$ where $x'\in D$ is the point indexed by $g$. By \cite[Theorem 4.4]{mendelson2016upper} we have that for any $s,t\ge 1,$ it holds with probability at least $1-2e^{-c_1 s^2 (\E[\sup_{x \in D} u(x)])^2}-2e^{-c_1N t^2}$ that 
    \begin{equation}\label{eq:lemma45_aux1}
    \sup_{x\in D} |\widehat{k}(x,x')-k(x,x')|\lesssim \frac{st \| \xi\|_{\psi_2} \,\E[\sup_{x \in D} u(x)]}{\sqrt{N}}\le \frac{st \,\E[\sup_{x \in D} u(x)]}{\sqrt{N}},
    \end{equation}
    where the last inequality follows by the fact that $\| \xi \|_{\psi_2} \le \sqrt{k(x',x')} \le \sup_{x \in D} \sqrt{k(x,x)}=1.$  We consider three cases:
 
      \emph{Case 1:} 
    If $\E [\sup_{x \in D} u(x)]<\frac{1}{\sqrt{N}}$, then $\rho_N=\frac{c_0}{N}<1$. We take 
    \[
    s=\frac{c_0}{\sqrt{N}\,\E[\sup_{x \in D} u(x)]}> 1,
    \qquad t=1,
    \]
    and then \eqref{eq:lemma45_aux1} implies that it holds with probability at least $1-2e^{-c_1 c_0^2/N}-2e^{-c_1N}\overset{\text{(i)}}{\ge} 1-4e^{-c_1 c_0^2/N}=1-4e^{-c_1N \rho_N^2 }$ that 
    \[
    \sup_{x\in D} |\widehat{k}(x,x')-k(x,x')|\lesssim \frac{st \,\E[\sup_{x \in D} u(x)]}{\sqrt{N}}=\frac{c_0}{N}=\rho_N,
    \]
    where (i) follows since $c_0<N$ by assumption. 
  
   \emph{Case 2:} If $\frac{1}{\sqrt{N}}\le \E[\sup_{x \in D} u(x)]\le\sqrt{N}$, then $\rho_N=\frac{c_0}{\sqrt{N}}\E[\sup_{x \in D} u(x)]$. In this case, if $\rho_N=\frac{c_0}{\sqrt{N}}\E[\sup_{x \in D} u(x)]>1$, we take
    \[
    s=\sqrt{\frac{c_0\sqrt{N}}{\E[\sup_{x \in D} u(x)]}}\ge 1,
    \qquad 
    t=\sqrt{\frac{c_0}{\sqrt{N}}\E[\sup_{x \in D} u(x)]}>1,
    \]
    and then \eqref{eq:lemma45_aux1} implies that it holds  with probability at least $1-4e^{-c_1c_0\sqrt{N}\E[\sup_{x \in D} u(x)]}=1-4e^{-c_1N\rho_N}$ that
    \[
    \sup_{x\in D} |\widehat{k}(x,x')-k(x,x')|\lesssim \frac{st \,\E[\sup_{x \in D} u(x)]}{\sqrt{N}}=\frac{c_0}{\sqrt{N}}\E[\sup_{x \in D} u(x)]=\rho_N.
    \]
     If $\rho_N=\frac{c_0}{\sqrt{N}}\E[\sup_{x \in D} u(x)]\le 1$, then we take $s=c_0\ge 1$ and  $t=1$, and \eqref{eq:lemma45_aux1} implies that, with probability at least $$1-2e^{-c_1c_0^2 (\E[\sup_{x \in D} u(x)])^2}-2e^{-c_1N}\ge 1-4e^{-c_1c_0^2 (\E[\sup_{x \in D} u(x)])^2}=1-4e^{-c_1N\rho_N^2},$$ it holds that
    \[
    \sup_{x\in D} |\widehat{k}(x,x')-k(x,x')|\lesssim \frac{st \,\E[\sup_{x \in D} u(x)]}{\sqrt{N}}=\frac{c_0}{\sqrt{N}}\E[\sup_{x \in D} u(x)]=\rho_N.
    \]
    
   \emph{Case 3:} If $\E[\sup_{x \in D} u(x)]>\sqrt{N}$, then $\rho_N=\frac{c_0}{N}(\E[\sup_{x \in D} u(x)])^2>1.$ We take
    \[
    s=\sqrt{c_0}\ge 1,
    \qquad 
    t=\sqrt{c_0}\frac{\E [\sup_{x \in D} u(x)]}{\sqrt{N}}>1,
    \]
    and \eqref{eq:lemma45_aux1} implies that it holds with probability at least $1-4e^{-c_1c_0 (\E[\sup_{x \in D} u(x)])^2}=1-4e^{-c_1N\rho_N}$ that 
    \[
    \sup_{x\in D} |\widehat{k}(x,x')-k(x,x')|\lesssim \frac{st \,\E[\sup_{x \in D} u(x)]}{\sqrt{N}}=\frac{c_0}{N}(\E[\sup_{x \in D} u(x)])^2=\rho_N.
    \]
    Combining the three cases above gives the desired result.   
\end{proof}

\subsection{Proof of Theorem \ref{thm:mainresult}}\label{ssec:proofmainresult1}

Before proving Theorem \ref{thm:mainresult}, the next result establishes moment and concentration bounds for the estimator $\widehat{\rho}_N$ of the thresholding parameter $\rho_N.$
\begin{lemma}\label{lemma:hat_rhoN_rhoN} 
   Under the setting of Theorem \ref{thm:mainresult}, it holds that 
\begin{enumerate}[label=(\Alph*)]
    \item For any $p\ge 1$, $
   \E \big[\, \widehat{\rho}_N^{\,p}\big] \lesssim_p \rho_N^p$.
   \item For any $t\in (0,1)$,
   \begin{align}
\mathbb{P}\left[\,\widehat{\rho}_N<t\rho_N\right]&\le 2\,e^{-\frac{1}{2}(1-\sqrt{t})^2 N (\E [\sup_{x \in D} u(x)])^2}\mathbf{1}\big\{\E[\sup_{x \in D} u(x)]\ge 1/\sqrt{N}\big\}\label{eq:cor42_aux1}\\
   &\le 2\,e^{-\frac{1}{2}(1-\sqrt{t})^2 N (\rho_N \land \,\rho_N^2)}\label{eq:cor42_aux2}.
   \end{align}
\end{enumerate}
\end{lemma}

The proof of Lemma \ref{lemma:hat_rhoN_rhoN} can be found in Appendix \ref{app:A}.

\begin{proof}[Proof of Theorem \ref{thm:mainresult}.]
As shown in Lemma \ref{lemma:operator_norm_bound} in Appendix \ref{app:B}, the operator norm can be upper bounded as 
\[
\|\widehat{\CC}_{\widehat{\rho}_N}-\CC\| \le \sup_{x\in D}\int_{D} \,\big|\widehat{k}_{\widehat{\rho}_N}(x,x')-k(x,x')\big| \,dx'.
\]

Let $\Omega_x:=\left\{x'\in D: |k(x,x')|\ge \widehat{\rho}_N\right\}$ and let $\Omega^c_x$ be its complement. Then, we have 
\begin{gather}\label{eq:thm22_main}
\begin{split}
&\E \|\widehat{\CC}_{\widehat{\rho}_N}-\CC\|^p \le \E\left[\bigg(\sup_{x\in D}\int_{D} \ \big|\widehat{k}_{\widehat{\rho}_N}(x,x')-k(x,x')\big| \ dx'\bigg)^p\right]\\
&\le 2^{p-1}\E\left[\bigg(\sup_{x\in D}\int_{\Omega_x} \ \big|\widehat{k}_{\widehat{\rho}_N}(x,x')-k(x,x')\big| \,dx'\bigg)^p\right] +  2^{p-1}\E\left[\bigg(\sup_{x\in D}\int_{\Omega^c_x} \ \big|\widehat{k}_{\widehat{\rho}_N}(x,x')-k(x,x')\big| \,dx'\bigg)^p\right] \\
&\lesssim_p \E\left[\bigg(\sup_{x\in D}\int_{\Omega_x} \ \big|\widehat{k}_{\widehat{\rho}_N}(x,x')-k(x,x')\big| \, dx'\bigg)^p\right]+\E\left[\bigg(\sup_{x\in D}\int_{\Omega^c_x} \ |k(x,x')| \mathbf{1}\big\{|\widehat{k}(x,x')|< \widehat{\rho}_N\big\} \, dx' \bigg)^p\right]\\
&\hspace{0.3cm} +\E\left[\bigg(\sup_{x\in D}\int_{\Omega^c_x} \ |\widehat{k}(x,x')-k(x,x')| \mathbf{1}\big\{|\widehat{k}(x,x')|\ge \widehat{\rho}_N\big\}\mathbf{1}\big\{|\widehat{k}(x,x')-k(x,x')|< 4|k(x,x')|\big\} \,dx'\bigg)^p\right]\\
&\hspace{0.3cm} +\E\left[\bigg(\sup_{x\in D}\int_{\Omega^c_x} \ |\widehat{k}(x,x')-k(x,x')| \mathbf{1}\big\{|\widehat{k}(x,x')|\ge \widehat{\rho}_N\big\} \mathbf{1}\big\{|\widehat{k}(x,x')-k(x,x')|\ge 4|k(x,x')|\big\} \,dx'\bigg)^p\right]\\
&=:I_1+I_2+I_3+I_4,
\end{split}
\end{gather}
where in the second inequality we used that $|a+b|^p \le 2^{p-1} (|a|^p + |b|^p),$ which follows directly from the convexity of $f(x) = |x|^p$ for $p \ge 1.$
We next bound the four terms $\{I_i\}_{i=1}^4$. To ease notation, we define
\[
\|\widehat{k}-k\|_{\max}:=\sup_{x,x'\in D} \,|\widehat{k}(x,x')-k(x,x')|.
\]

For $I_1$, using that
\begin{align*}
\big|\widehat{k}_{\widehat{\rho}_N}(x,x')-k(x,x')\big|
\le \big|\widehat{k}_{\widehat{\rho}_N}(x,x')-\widehat{k}(x,x')\big|+\big|\widehat{k}(x,x')-k(x,x')\big|\le \widehat{\rho}_N+\|\widehat{k}-k\|_{\max},
\end{align*}
we have
\[
I_1=\E\left[\bigg(\sup_{x\in D}\int_{\Omega_x} \ \big|\widehat{k}_{\widehat{\rho}_N}(x,x')-k(x,x')\big| \, dx'\bigg)^p\right] \le \E\left[\Big(\sup_{x\in D}\mathrm{Vol}(\Omega_x)\Big)^p\left( \widehat{\rho}_N+\|\widehat{k}-k\|_{\max}\right)^p\right],
\]
where $\mathrm{Vol}(\Omega_x)$ denotes the Lebesgue measure of $\Omega_x.$
Notice that
\[
R_q^q\ge \sup_{x\in D}\int_D |k(x,x')|^q dx' \ge \sup_{x\in D}\int_{\Omega_x}|k(x,x')|^q dx'\ge \sup_{x\in D} \int_{\Omega_x}\widehat{\rho}_N^{\ q}dx'=\widehat{\rho}_N^{\ q}\sup_{x\in D} \mathrm{Vol}(\Omega_x).
\]
Combining this bound with the trivial bound $\sup_{x} \mathrm{Vol}(\Omega_x)\le \mathrm{Vol}(D)=1$ gives
\[
\sup_{x\in D} \mathrm{Vol}(\Omega_x)\le R^q_q \widehat{\rho}_N^{\,-q}\land 1.
\]
Therefore, by Cauchy-Schwarz, we have that
\begin{align}\label{eq:I_1_main}
I_1
&\le \E\left[ \big(R^q_q \widehat{\rho}_N^{\,-q}\land 1\big)^p (\widehat{\rho}_N+\|\widehat{k}-k\|_{\max})^p\right] \nonumber\\
&\le \sqrt{\E \Big[\big(R^q_q \widehat{\rho}_N^{\,-q}\land 1\big)^{2p}\Big]\ \E \Big[\big(\widehat{\rho}_N+\|\widehat{k}-k\|_{\max}\big)^{2p}\Big]}.
\end{align}

Using Lemma \ref{lemma:hat_rhoN_rhoN} and Corollary \ref{cor:max_norm_bound} yields that
\begin{align}\label{eq:I_1_aux1}
\E \Big[\big(\widehat{\rho}_N+\|\widehat{k}-k\|_{\max}\big)^{2p}\Big]\lesssim_p \E \big[\big(\widehat{\rho}_N\big)^{2p}\big] + \E \Big[\|\widehat{k}-k\|_{\max}^{2p}\Big]\lesssim_p \rho_N^{2p}.
\end{align}
On the other hand,
\begin{align*}
\E \Big[\big(R^q_q \widehat{\rho}_N^{\,-q}\land 1\big)^{2p}\Big]&= R_q^{2pq}\, \E \left[\widehat{\rho}_N^{\,-2pq}\land R_q^{-2pq}\right]= R_q^{2pq}\int_0^{\infty} \mathbb{P}\Big[\big(\widehat{\rho}_N^{\,-2pq}\land R_q^{-2pq}\big)>t\Big] dt\\
&=R_q^{2pq}\int_0^{R_q^{-2pq}}\mathbb{P}\left[\widehat{\rho}_N^{\,-2pq}>t\right] dt=2pq R_q^{2pq}\int_{R_q}^{\infty} \mathbb{P}\left[\,\widehat{\rho}_N<t\right]\ t^{-2pq-1} dt.
\end{align*}
If $R_q>\rho_N$, then 
\begin{gather}\label{eq:I_1_aux2}
\begin{split}
\E \Big[\big(R^q_q \widehat{\rho}_N^{\,-q}\land 1\big)^{2p}\Big]
&\le 2pq R_q^{2pq}\int_{\rho_N}^{\infty} t^{-2pq-1} dt= R_q^{2pq}\rho_{N}^{-2pq}.
\end{split}
\end{gather}
If $R_q<\rho_N$, then
\begin{align}\label{eq:I_1_aux3}
&\E \Big[\big(R^q_q \widehat{\rho}_N^{\,-q}\land 1\big)^{2p}\Big]
=2pq R_q^{2pq}\left(\int_{\rho_N}^{\infty}+\int_{R_q}^{\rho_N}\right) \mathbb{P}\left[\widehat{\rho}_N<t\right]\ t^{-2pq-1} dt \nonumber\\
&\le 2pq R_q^{2pq}\int_{\rho_N}^{\infty} t^{-2pq-1} dt+2pq R_q^{2pq}\int_{R_q}^{\rho_N} \mathbb{P}\left[\widehat{\rho}_N<t\right]\ t^{-2pq-1} dt \nonumber\\
&=R_q^{2pq}\rho_{N}^{-2pq}+2pq R_q^{2pq}\rho_N^{-2pq}\int_{R_q\rho_N^{-1}}^{1} \mathbb{P}\left[\widehat{\rho}_N<t\rho_N\right]\ t^{-2pq-1} dt \nonumber\\
&\overset{\text{(i)}}{\le} R_q^{2pq}\rho_{N}^{-2pq}+2pq R_q^{2pq}\rho_N^{-2pq}\int_{R_q\rho_N^{-1}}^{1} 2\exp\Big(-\frac{1}{2}(1-\sqrt{t})^2 N (\rho_N\land \,\rho_N^2)\Big) t^{-2pq-1}dt \nonumber\\
&\overset{\text{(ii)}}{=}R_q^{2pq}\rho_{N}^{-2pq}\Bigg[1+8pq\int_{0}^{\sqrt{N (\rho_N\land \,\rho_N^2)}(1-\sqrt{R_q\rho_N^{-1}})}  \frac{\big(N (\rho_N\land \,\rho_N^2)\big)^{2pq}\exp(-\frac{1}{2}t^2)}{\big(\sqrt{N (\rho_N\land \,\rho_N^2)}-t\big)^{4pq+1}}\ dt\Bigg] \nonumber\\
&\overset{\text{(iii)}}{\lesssim} R_q^{2pq}\rho_N^{-2pq}+R_q^{2pq}\rho_N^{-2pq}\cdot 8pq \bigg(\frac{2R_q^{-2pq}\rho_N^{2pq}}{4pq}e^{-\frac{1}{8}N (\rho_N\land \,\rho_N^2) \big(1-\sqrt{R_q\rho_N^{-1}}\big)^2}+\frac{2^{4pq}}{4pq}\bigg) \nonumber\\
&\lesssim_p R_q^{2pq}\rho_N^{-2pq}+ e^{-\frac{1}{8}N (\rho_N\land \,\rho_N^2) \big(1-\sqrt{R_q\rho_N^{-1}}\big)^2} \nonumber\\
&\overset{\text{(iv)}}{\lesssim}_p R_q^{2pq}\rho_N^{-2pq}+e^{-c N(\rho_N\land\, \rho_N^2)},
\end{align}
where (i) follows from Lemma \ref{lemma:hat_rhoN_rhoN}, (ii) follows by a change of variable, and (iii) follows by applying Lemma \ref{lemma:technique} in Appendix \ref{app:C} with $\alpha=\sqrt{N (\rho_N\land \,\rho_N^2)}$ and $\beta= \sqrt{N (\rho_N\land \,\rho_N^2)} \sqrt{R_q\rho_N^{-1}}$. To prove (iv), notice that if $R_q\le \frac{1}{4}\rho_N$, then $\big|1-\sqrt{R_q\rho_N^{-1}}\big|>\frac{1}{2}$ and (iv) holds; if $\frac{1}{4}\rho_N< R_q<\rho_N$, then
\[
e^{-\frac{1}{8}N (\rho_N\land \,\rho_N^2) \big(1-\sqrt{R_q\rho_N^{-1}}\big)^2}\le 1< 16^p R_q^{2p}\rho_N^{-2p}\le 16^p R_q^{2pq}\rho_N^{-2pq}.
\]

Combining the inequalities \eqref{eq:I_1_main}, \eqref{eq:I_1_aux1}, \eqref{eq:I_1_aux2}, and \eqref{eq:I_1_aux3} gives that
\[
I_1\le \sqrt{\E \Big[\big(R^q_q \widehat{\rho}_N^{\,-q}\land 1\big)^{2p}\Big]\ \E \Big[\big(\widehat{\rho}_N+\|\widehat{k}-k\|_{\max}\big)^{2p}\Big]}\lesssim_p R_q^{pq}\rho_N^{p(1-q)}+\rho_N^p e^{-c N(\rho_N\land\, \rho_N^2)}.
\]

For $I_2$ and $I_3$,
\[
\begin{split}
I_2+I_3&=\E\left[\bigg(\sup_{x\in D}\int_{\Omega^c_x} \ \left|k(x,x')\right| \,\mathbf{1}\big\{|\widehat{k}(x,x')|< \widehat{\rho}_N\big\} \,dx'\bigg)^p\right] \\
&\ +\E\left[\bigg(\sup_{x\in D}\int_{\Omega^c_x} \ |\widehat{k}(x,x')-k(x,x')| \,\mathbf{1}\big\{|\widehat{k}(x,x')|\ge \widehat{\rho}_N\big\}\,\mathbf{1}\big\{|\widehat{k}(x,x')-k(x,x')|< 4|k(x,x')|\big\} \,dx'\bigg)^p\right]\\
&\lesssim \E\left[\bigg(\sup_{x\in D}\int_{\Omega^c_x} \ \left|k(x,x')\right| \,dx'\bigg)^p\right]
= \E\left[\bigg(\widehat{\rho}_N \sup_{x\in D}\int_{\Omega^c_x} \left(\frac{|k(x,x')|}{\widehat{\rho}_N}\right) \, dx'\bigg)^p\right]\\
&\overset{\text{(i)}}{\le} \E\left[\bigg(\widehat{\rho}_N\sup_{x\in D}\int_{\Omega^c_x} \left(\frac{|k(x,x')|}{\widehat{\rho}_N}\right)^{q} dx'\bigg)^p\right]\le \E \left[R_q^{pq} \widehat{\rho}_N^{\,p(1-q)}\right]\overset{\text{(ii)}}{\lesssim}_p R_q^{pq} \rho_N^{p(1-q)},
\end{split}
\]
where (i) follows since $q\in (0,1)$ and $|k(x,x')|<\widehat{\rho}_N$ for $x'\in \Omega_x^c$. To prove (ii), we notice that if $p(1-q)\le 1$, then using Jensen's inequality and Lemma \ref{lemma:hat_rhoN_rhoN} yields that $\E[\,\widehat{\rho}_N^{\,p(1-q)}]\le (\E[\,\widehat{\rho}_N])^{p(1-q)}\lesssim_p  \rho_N^{p(1-q)}$. If $p(1-q)>1$, Lemma \ref{lemma:hat_rhoN_rhoN} implies that $\E[\,\widehat{\rho}_N^{\,p(1-q)}]\lesssim_p \rho_N^{\,p(1-q)}$.

\vspace{1em}

For $I_4$,
\begin{align*}
I_4&=\E\left[\bigg(\sup_{x\in D}\int_{\Omega^c_x} \ |\widehat{k}(x,x')-k(x,x')| \mathbf{1}\big\{|\widehat{k}(x,x')|\ge \widehat{\rho}_N\big\} \mathbf{1}\big\{|\widehat{k}(x,x')-k(x,x')|\ge 4|k(x,x')|\big\} \,dx'\bigg)^p\right]\\
&\overset{\text{(i)}}{\le} \E\left[\bigg(\sup_{x\in D}\int_{\Omega^c_x} \ |\widehat{k}(x,x')-k(x,x')| \mathbf{1}\big\{|\widehat{k}(x,x')|\ge \widehat{\rho}_N\big\} \mathbf{1}\big\{|\widehat{k}(x,x')-k(x,x')|\ge \frac{2}{3}\widehat{\rho}_N \big\} \,dx'\bigg)^p\right] \\
&\le \E\left[\bigg(\sup_{x\in D}\int_D \sup_{x\in D}\,|\widehat{k}(x,x')-k(x,x')|\, \mathbf{1}\big\{\sup_{x\in D}|\widehat{k}(x,x')-k(x,x')|\ge \frac{2}{3}\widehat{\rho}_N\big\}\,dx'\bigg)^p\right]\\
&=\E\left[\bigg(\int_D \sup_{x\in D}|\widehat{k}(x,x')-k(x,x')| \,\mathbf{1}\big\{\sup_{x\in D}|\widehat{k}(x,x')-k(x,x')|\ge \frac{2}{3}\widehat{\rho}_N\big\} \,dx'\bigg)^p\right]\\
&\le\E\left[\bigg(\|\widehat{k}-k\|_{\max}\int_D \mathbf{1}\big\{\sup_{x\in D}|\widehat{k}(x,x')-k(x,x')|\ge \frac{2}{3}\widehat{\rho}_N\big\} \,dx'\bigg)^p\right]  \\
&\le \Big(\E \big[\|\widehat{k}-k\|_{\max}^{2p}\big]\Big)^{1/2} \left(\E\bigg[ \Big(\int_D \mathbf{1}\big\{\sup_{x\in D}|\widehat{k}(x,x')-k(x,x')|\ge \frac{2}{3}\widehat{\rho}_N\big\} \,dx'\Big)^{2p}\bigg]\right)^{1/2}\\
&\overset{\text{(ii)}}{\le} \Big(\E \big[\|\widehat{k}-k\|_{\max}^{2p}\big]\Big)^{1/2} \left(\E\Big[\int_D \mathbf{1}\big\{\sup_{x\in D}|\widehat{k}(x,x')-k(x,x')|\ge \frac{2}{3}\widehat{\rho}_N\big\} \,dx'\Big]\right)^{1/2}\\
&=\Big(\E \big[\|\widehat{k}-k\|_{\max}^{2p}\big]\Big)^{1/2} \left(\int_{D}\mathbb{P}\left[ \sup_{x\in D}|\widehat{k}(x,x')-k(x,x')|\ge \frac{2}{3}\widehat{\rho}_N\right] dx'\right)^{1/2},
\end{align*}
where 
(i) follows since  $|\widehat{k}(x,x')-k(x,x')|\ge 4|k(x,x')|$ implies that $|\widehat{k}(x,x')|\ge 3|k(x,x')|,$ and therefore if  $|\widehat{k}(x,x')-k(x,x')|\ge 4|k(x,x')|$ and $|\widehat{k}(x,x')|\ge \widehat{\rho}_N,$ then it holds that
\[
|\widehat{k}(x,x')-k(x,x')|\ge |\widehat{k}(x,x')|-|k(x,x')|\ge \frac{2}{3}|\widehat{k}(x,x')|\ge \frac{2}{3}\widehat{\rho}_N.
\]
To prove (ii), note that $p\ge 1$ and
$
\int_D \mathbf{1}\left\{\sup_{x\in D}|\widehat{k}(x,x')-k(x,x')|\ge \frac{2}{3}\widehat{\rho}_N\right\} \,dx'\le 1.
$
Next, notice that
\begin{align*}
\mathbb{P}\left[ \sup_{x\in D}|\widehat{k}(x,x')-k(x,x')|\ge \frac{2}{3}\widehat{\rho}_N\right]&= \mathbb{P}\left[ \frac{2}{3}(\,\rho_N-\widehat{\rho}_N) +\sup_{x\in D}|\widehat{k}(x,x')-k(x,x')|\ge \frac{2}{3}\rho_N\right]\\
&\le \mathbb{P}\left[\sup_{x\in D}|\widehat{k}(x,x')-k(x,x')|\ge \frac{1}{3}\rho_N\right]+\mathbb{P}\Big[\rho_N-\widehat{\rho}_N \ge\frac{1}{2}\rho_N \Big].
\end{align*}
Lemma \ref{lemma:hat_rhoN_rhoN} then implies that
\[
\mathbb{P}\Big[\rho_N-\widehat{\rho}_N \ge\frac{1}{2}\rho_N \Big]= \mathbb{P}\Big[\widehat{\rho}_N \le\frac{1}{2}\rho_N \Big]\lesssim e^{-c_1 N(\rho_N\land\, \rho_N^2)},
\]
and  Proposition \ref{prop:single_supremum} gives that
$$
\mathbb{P}\left[\sup_{x\in D}|\widehat{k}(x,x')-k(x,x')|\ge \frac{1}{3}\rho_N\right]\lesssim e^{-c_2 N(\rho_N\land\, \rho_N^2)}.$$
Moreover, Corollary \ref{cor:max_norm_bound} yields that $\Big(\E\big[ \|\widehat{k}-k\|_{\max}^{2p}\big]\Big)^{1/2}\lesssim_p \rho_N^p$. Therefore,
\begin{align*}
I_4&\le \Big(\E \big[\|\widehat{k}-k\|_{\max}^{2p}\big]\Big)^{1/2} \left(\int_{D}\mathbb{P}\left[ \sup_{x\in D}|\widehat{k}(x,x')-k(x,x')|\ge \frac{2}{3}\widehat{\rho}_N\right] dx'\right)^{1/2}
\lesssim_p \rho_N^p e^{-c N(\rho_N\land\, \rho_N^2)}.
\end{align*}

Combining \eqref{eq:thm22_main} with the estimates of $I_1, I_2,I_3$, and $I_4$ gives that 
\[
\E \|\widehat{\CC}_{\widehat{\rho}_N}-\CC\|^p\lesssim_p I_1+I_2+I_3+I_4\lesssim_p R_q^{pq} \rho_N^{p(1-q)}+\rho_N^p e^{-c N(\rho_N\land\, \rho_N^2)},
\]
and hence
\[
\big[\E \|\widehat{\CC}_{\widehat{\rho}_N}-\CC\|^p\big]^{\frac{1}{p}} \lesssim_p R_q^q\rho_{N}^{1-q}+\rho_Ne^{-\frac{c}{p} N(\rho_N\land\, \rho_N^2)}. \qedhere
\] 
\end{proof}

\section{Small Lengthscale Regime}\label{sec:smalllengthscale}
This section studies thresholded estimation of covariance operators under the small lengthscale regime formalized in Assumption \ref{assumption:main2}. We first present three lemmas which establish the sharp scaling of the $L^q$-sparsity level, the operator norm of the covariance operator, and the suprema of Gaussian fields in the small lengthscale regime. 
Combining these lemmas and Theorem \ref{thm:mainresult}, we then prove Theorem \ref{thm:smalllengthscale}. Throughout this section, we use the notation ``$(\mathcal{B}), \, \lambda\to 0$'' to indicate that there is a universal constant $\lambda_0>0$ such that if $\lambda<\lambda_0$, the conclusion $(\mathcal{B})$ holds.

The following result establishes the scaling of the $L^q$-sparsity level in the small lengthscale regime.
\begin{lemma}\label{lemma:R_q^q_bound}
Under Assumption \ref{assumption:main2}, it holds that 
\begin{equation*}
    \sup_{x\in D}\int_{D} |k(x,x')|^q dx' \asymp \lambda^{d}A(d)\int_0^{\infty} k_1(r)^q r^{d-1} dr, \quad \lambda\to 0,
\end{equation*}
where $A(d)$ denotes the surface area of the unit sphere in $\R^d$.
\end{lemma}
\begin{proof}
 Using $k_{\lambda}(r)=k_{1}(\lambda^{-1}r)$, we have that
\begin{align}\label{eq:Lq_Sparsity_aux1}
\sup_{x\in D}\int_{D}& |k(x,x')|^q dx' \ge \int_{D\times D} k(x,x')^q dxdx'
=\int_{[0,1]^d\times[0,1]^d} k_{\lambda}(|x-x'|)^q \,dxdx' \nonumber\\
&=\int_{[0,1]^d\times[0,1]^d} k_1(\lambda^{-1}|x-x'|)^q \,dxdx' 
=\lambda^{2d}\int_{[0,\lambda^{-1}]^d\times [0,\lambda^{-1}]^d} k_1(|x-x'|)^q\,dxdx'\nonumber\\
&\overset{(\text{i})}{=} \lambda^{2d}\int_{[-\lambda^{-1},\lambda^{-1}]^d} k_1(|w|)^q \prod_{j=1}^{d}(\lambda^{-1}-|w_j|) \,dw \nonumber\\
&=\lambda^{d}\int_{[-\lambda^{-1},\lambda^{-1}]^d} k_1(|w|)^q \prod_{j=1}^{d}(1-\lambda|w_j|) \,dw\nonumber\\
&\overset{(\text{ii})}{\asymp} \lambda^{d}\int_{\R^d} k_1(|w|)^q \,dw  \overset{(\text{iii})}{=}\lambda^{d}A(d)\int_0^{\infty} k_1(r)^q r^{d-1}\,dr,\quad \lambda\to 0,
\end{align}
where (i) follows by a change of variables $w=x-x', z=x+x'$ and integrating $z$, (ii) follows by dominated convergence as $\lambda\to 0$, and (iii) follows from the polar coordinate transform in $\R^d$. On the other hand,
\[
\begin{split}
    \sup_{x\in D}\int_{D} |k(x,x')|^q dx'&\le \sup_{x\in D}\int_{\R^d} k(|x-x'|)^q dx'\\
    &=\int_{\R^d} k(|x'|)^q \,dx'=\lambda^{d}\int_{\R^d} k_1(|x'|)^q\,dx'=\lambda^{d}A(d)\int_0^{\infty} k_1(r)^q r^{d-1}\,dr,
\end{split}
\]
which concludes the proof.
\end{proof}

Next, we establish the scaling of the operator norm of the covariance operator.
\begin{lemma}\label{lemma:operator_norm_C}
    Under Assumption \ref{assumption:main2}, it holds that 
    \[
    \|\CC\|\asymp \lambda^{d}A(d)\int_0^{\infty} k_1(r)r^{d-1}dr,\quad \lambda\to 0,
    \]
    where $A(d)$ denotes the surface area of the unit sphere in $\R^d$.
\end{lemma}
\begin{proof}
First, the operator norm can be upper bounded by
\[
\|\CC \| \le \sup_{x\in D} \int_{D} |k(x,x')|dx'\asymp \lambda^{d}A(d)\int_0^{\infty} k_1(r)r^{d-1}dr,\quad \lambda\to 0,
\]
where the last step follows by Lemma \ref{lemma:R_q^q_bound}. 
    
For the lower bound, taking the test function $\psi(x) \equiv 1$ yields that
$$
\begin{aligned}
\|\CC\| &=\sup_{\|\psi\|_{L^2}=1}\Big(\int_{D}\Big(\int_D k(x,x')\psi(x') dx'\Big)^2 dx\Big)^{1/2}
 \ge \Big(\int_D\Big(\int_D k(x,x') d x^{\prime}\Big)^2 d x\Big)^{1/2} \\
&\overset{\text{(i)}}{\ge} \frac{1}{\sqrt{\mathrm{Vol}(D)}}\int_{D\times D} k(x,x')dxdx'
\overset{\text{(ii)}}{=}\int_{D\times D} k(x,x')dxdx'\\
&\overset{\text{(iii)}}{\asymp} \lambda^{d}A(d)\int_0^{\infty} k_1(r) r^{d-1}\,dr,\quad \lambda\to 0,
\end{aligned}
$$
where (i) follows by Cauchy-Schwarz inequality, (ii) follows since $\mathrm{Vol}(D) =1$ for $D=[0,1]^d$, and (iii) follows from \eqref{eq:Lq_Sparsity_aux1} with $q=1$. This completes the proof.
\end{proof}

Finally, we establish the scaling of the suprema of Gaussian fields in the small lengthscale regime.
\begin{lemma}\label{lemma:sup_bound}
Under Assumption \ref{assumption:main2} \ref{assumption:fieldandkernel}, it holds that
\[
\E \left[ \,\sup_{x \in D} u(x) \right] \asymp \sqrt{d}\int_0^{\infty}\sqrt{k(0)-k(c\sqrt{d}e^{-t^2})} \ dt,
\]
where $c$ is an absolute constant. Furthermore, if Assumption \ref{assumption:main2} \ref{assumption:smalllengthscalelimit} also holds, then
\[
\E \left[ \,\sup_{x \in D} u(x) \right] \asymp \sqrt{k(0)d\log \Big(\frac{\sqrt{d}}{s\lambda}\Big)}\,\, ,\quad  \lambda \to 0,
\]
where $s>0$ is the unique solution of $k_{1}(s)=\frac{1}{2}k(0),$ which is independent of $\lambda$.

\end{lemma}
\begin{proof}
By Fernique's theorem \cite{fernique1975regularite} and the discussion in \cite[Theorem 6.19]{van2014probability}, for the stationary Gaussian random field $u,$ it holds that 
\begin{equation}\label{eq:Fernique}
     \E \left[ \,\sup_{x \in D} u(x) \right]\asymp \int_0^{\infty} \sqrt{\log\mathcal{M}(D,\mathsf{d},\varepsilon)} \ d\varepsilon,
\end{equation}
where $\mathcal{M}(D,\mathsf{d},\varepsilon)$ denotes the smallest cardinality of an $\varepsilon$-net of $D$ in the canonical metric $\mathsf{d}$ given by  
\[\mathsf{d}(x,x') :=\big(\E [(u(x)-u(x'))^2]\big)^{1/2}=\sqrt{2k(0)-2k(|x-x'|)}< \sqrt{2k(0)},\quad x,x'\in D.
\] 
Since the Gaussian field is assumed to be isotropic under Assumption \ref{assumption:main2} \ref{assumption:fieldandkernel}, the field is necessarily stationary. Consequently, Fernique's bound implies that  $\mathcal{M}(D,\mathsf{d},\varepsilon) = 1$ for $\varepsilon \ge \sqrt{2k(0)},$ and hence we can assume without loss of generality that $\varepsilon<\sqrt{2k(0)}$ in the rest of the proof. Next, notice that
\[
\mathsf{d}(x,x')=\sqrt{2k(0)-2k(|x-x'|)} \le \varepsilon \quad \Longleftrightarrow \quad |x-x'|\le k^{-1}(k(0)-\varepsilon^2/2),
\]
where $k^{-1}$ is the inverse function of $k$.
By the standard volume argument \cite[Proposition 4.2.12]{vershynin2018high},
\[
\begin{split}
\mathcal{M}(D,\mathsf{d},\varepsilon)&=\mathcal{M}\bigl(D,|\cdot|,k^{-1}(k(0)-\varepsilon^2/2)\bigr)\\
&\ge \left(\frac{1}{k^{-1}(k(0)-\varepsilon^2/2)}\right)^d\frac{\mathrm{Vol}(D)}{\mathrm{Vol}(B_2^d)}\ge \frac{1}{c_1 }\left(\frac{1}{k^{-1}(k(0)-\varepsilon^2/2)}\right)^d \left(\frac{d}{2\pi e}\right)^{d/2},
\end{split}
\]
where we used that $\mathrm{Vol}(D)=1$ and that, for the Euclidean unit ball $B_2^d,$ it holds that $\mathrm{Vol}(B_2^d)\le c_1 (2\pi e/d)^{d/2}$ for some absolute constant $c_1 >1$. On the other hand, using the fact that $D=[0,1]^d\subset \sqrt{d}B_2^d$, as well as $\mathcal{M}(B_2^d,|\cdot|,\varepsilon)\le (3/\varepsilon)^{d}$ for $\varepsilon\le 1$ \cite[Corollary 4.2.13]{vershynin2018high},
\[
\begin{split}
&\mathcal{M}(D,\mathsf{d},\varepsilon)=\mathcal{M} \bigl(D,|\cdot|,k^{-1}(k(0)-\varepsilon^2/2)\bigr)\\
&\le \mathcal{M}\bigl(B_2^d,|\cdot|, d^{-1/2} k^{-1}(k(0)-\varepsilon^2/2)\bigr)\le \left[\left(\frac{3}{k^{-1}(k(0)-\varepsilon^2/2)}\right)^d d^{\,d/2}\right] \lor 1, \quad \varepsilon < \sqrt{2 k(0)}.
\end{split}
\]
Therefore, \eqref{eq:Fernique} and the bounds we just established on the covering number $\mathcal{M}(D,\mathsf{d},\varepsilon)$ imply that
\[
\begin{split}
    \E \left[ \,\sup_{x \in D} u(x) \right]
    &\asymp \int_0^{\sqrt{2k(0)}}\sqrt{\log\left(\frac{1}{c_1 }\left(\frac{1}{k^{-1}(k(0)-\varepsilon^2/2)}\right)^d \left(\frac{d}{2\pi e}\right)^{d/2}\right) \vee 0 } \ d\varepsilon\\
    &\asymp \sqrt{d}\int_0^{\sqrt{2k(0)}}\sqrt{\log\left(\frac{c\sqrt{d}}{k^{-1}(k(0)-\varepsilon^2/2)} \right) \vee 0 } \ d\varepsilon.
\end{split}
\]
By a change of variable $t:=\sqrt{\log\left(\frac{c\sqrt{d}}{k^{-1}(k(0)-\varepsilon^2/2)}\right)}$, then $\varepsilon=\sqrt{2\left(k(0)-k(c\sqrt{d}e^{-t^2})\right)}$ and 
\[
\begin{split}
\E \left[ \,\sup_{x \in D} u(x) \right] &\asymp \sqrt{d}\int_0^{\infty} -t\, \frac{d}{dt}\left(\sqrt{k(0)-k(c\sqrt{d}e^{-t^2})}\right) \ dt\\
&= \sqrt{d}\left( -t\sqrt{k(0)-k(c\sqrt{d}e^{-t^2})} \,\Big|_{0}^{\infty}+\int_0^{\infty}\sqrt{k(0)-k(c\sqrt{d}e^{-t^2})} \ dt\right)\\
&=\sqrt{d}\int_0^{\infty}\sqrt{k(0)-k(c\sqrt{d}e^{-t^2})} \ dt,
\end{split}
\]
where in the last equality we used that $k(0)-k(c\sqrt{d}e^{-t^2})\asymp c\sqrt{d}e^{-t^2}$ as $t\to\infty$ since $k(r)$ is assumed to be differentiable at $r=0$. 

Furthermore, if Assumption \ref{assumption:main2}  \ref{assumption:smalllengthscalelimit} also holds, let $s>0$ be the unique solution of $k_{1}(s)=\frac{1}{2}k(0),$ which is independent of $\lambda$. Then, for $\lambda< c\sqrt{d}/s$,
\[
\begin{split}
   \E \left[ \,\sup_{x \in D} u(x) \right] &\asymp \sqrt{d}\int_0^{\infty}\sqrt{k(0)-k_{\lambda}(c\sqrt{d}e^{-t^2})} \ dt
    =\sqrt{d}\int_0^{\infty}\sqrt{k(0)-k_{1}(c\lambda^{-1}\sqrt{d}e^{-t^2})} \ dt\\
    &=\sqrt{d}\left[\int_{t<\sqrt{\log \big(\frac{c\sqrt{d}}{s \lambda }\big)}}+\int_{t>\sqrt{\log \big(\frac{c\sqrt{d}}{s \lambda }\big)}}\right] \sqrt{k(0)-k_{1}(c\lambda^{-1}\sqrt{d}e^{-t^2})} \ dt 
    =:I_1+I_2.
\end{split}
\]

For the first term $I_1$, we have
\begin{align*}
 \sqrt{\frac{k(0)d}{2}}\sqrt{\log \Big(\frac{c\sqrt{d}}{s\lambda}\Big)}
 \le 
 I_1 
 \le    
 \sqrt{k(0)d}\sqrt{\log \Big(\frac{c\sqrt{d}}{s\lambda}\Big)}.
\end{align*}
Therefore, for any $\lambda < c\sqrt{d}/s$, $
I_1\asymp \sqrt{k(0)d\log \Big(\frac{\sqrt{d}}{s\lambda}\Big)}$\hspace{0.005cm} .

To bound the second term $I_2$, we notice that there is some constant $M>0$ such that $k(0)-k_1(r)\le M\, r$ for $r\in [0,s]$, where $M$ is independent of $\lambda$. Therefore,  
\begin{align*}
    I_2&=\sqrt{d}\int_{t>\sqrt{\log \big(\frac{c\sqrt{d}}{s \lambda }\big)}} \sqrt{k(0)-k_{1}(c\lambda^{-1}\sqrt{d}e^{-t^2})} \ dt
    \le \sqrt{d}\int_{t>\sqrt{\log \big(\frac{c\sqrt{d}}{s \lambda }\big)}}
    \sqrt{M} \left(c\lambda^{-1}\sqrt{d}e^{-t^2}\right)^{\frac{1}{2}} \ dt \\
    &\lesssim d^{3/4}\lambda^{-\frac{1}{2}} \int_{t>\sqrt{\log \big(\frac{c\sqrt{d}}{s \lambda}\big)}} e^{-\frac{1}{2}t^2} \ dt
    \lesssim \sqrt{d}\,\left(\log \Big(\frac{c\sqrt{d}}{s\lambda}\Big)\right)^{-\frac{1}{2}}\to 0,\quad \lambda\to 0,
\end{align*}
where we used the tail bound of the Gaussian distribution $\int_x^{\infty} e^{-\frac{1}{2}t^2} dt \le \frac{1}{x}e^{-\frac{1}{2}x^2}$ for $x>0$. Since $I_2 \ge 0,$ we therefore have that $I_2\lesssim \sqrt{d} \left(\log \big(\frac{c\sqrt{d}}{\lambda}\big)\right)^{-\frac{1}{2}}\to 0$ as $\lambda \to 0$. 
Consequently,
\[
 \E \left[ \,\sup_{x \in D} u(x) \right] \asymp I_1+I_2\asymp \sqrt{k(0)d\log \Big(\frac{\sqrt{d}}{s\lambda}\Big)}\, \, ,\quad \lambda\to 0. \qedhere
\]
\end{proof}


\begin{remark}\label{remark:heuristicsmalllengthscale}
    Lemma \ref{lemma:sup_bound} admits a clear heuristic interpretation. Consider a uniform mesh $\mathcal{P}$ of the unit cube $D = [0,1]^d$ comprising $(1/\lambda)^d$ points that are distance $\lambda$ apart. For a random field $u(x)$ with lengthscale $\lambda,$ the values $u(x_i)$ and $u(x_j)$ at mesh points $x_i \neq x_j \in \mathcal{P}$ are roughly uncorrelated. Thus, $\{ u(x_i) \}_{i=1}^{\lambda^{-d}}$ are roughly i.i.d. univariate Gaussian random variables, and, for small $\lambda,$ we may approximate
    \begin{equation*}
        \E \left[ \,\sup_{x \in D} u(x) \right] \approx   \E \left[ \,\sup_{x_i \in \mathcal{P}} u(x_i) \right] \approx \sqrt{\log (\lambda^{-d})}. 
    \end{equation*}
 This heuristic derivation matches the scaling of the expected supremum with $\lambda$ in Lemma \ref{lemma:sup_bound}.  \qed
\end{remark}

We are now ready to prove Theorem \ref{thm:smalllengthscale}.
\begin{proof}[Proof of Theorem \ref{thm:smalllengthscale}]
In this proof we treat $d$ as a constant. Notice that under Assumption \ref{assumption:main1} \ref{assumption:magnitude} and Assumption \ref{assumption:main2} \ref{assumption:fieldandkernel}, it holds that $\mathrm{Tr}(\CC)= \int_{D}k(x,x)\,dx=k(0) \mathrm{Vol}(D)=1$. Moreover, Lemma \ref{lemma:operator_norm_C} shows that $\|\CC\|\asymp \lambda^d$ as $\lambda\to 0$. Plugging the former two results into \eqref{eq:Koltchinksiibound} yields the bound \eqref{eq:smalllengthscalewithoutthresholding}.
For the thresholded estimator, we apply Theorem \ref{thm:mainresult} with an appropriate choice of the constant $c_0 \in [1, \sqrt{N}].$
By Lemma \ref{lemma:sup_bound}, $\E [\sup_{x \in D} u(x)]\asymp \sqrt{\log(\lambda^{-d})}$ as $\lambda\to 0$. We assume that $N\ge c_0^2\,(\E[\sup_{x \in D} u(x)])^2\asymp \log (\lambda^{-d})$, so that the thresholding parameter satisfies
\[
\rho_N=\frac{c_0}{\sqrt{N}} \E\Big[\sup_{x\in D} u(x)\Big]\le 1.
\]
It follows that
\begin{gather}\label{eq:thm2_aux1}
\begin{split}
\rho_Ne^{-cN(\rho_N\land\, \rho_N^2)}=\rho_N e^{-cN\rho_N^2}&=\rho_N e^{-cc_0^2(\E[\sup_{x \in D} u(x)])^2}\\
&=\rho_N e^{-cc'c_0^2 d\log (1/\lambda)}=\rho_N \lambda^{cc'c_0^2d}\le \rho_{N}^{1-q}\lambda^{cc'c_0^2d},
\end{split}
\end{gather}
where $c'$ is an absolute constant. On the other hand, using Lemma \ref{lemma:R_q^q_bound} we have that
\begin{equation}\label{eq:thm2_aux2}
R_q^q\rho_N^{1-q}\asymp \rho_N^{1-q}\lambda^d  A(d)\int_0^{\infty} k_1(r)^q r^{d-1}dr.
\end{equation}
Comparing \eqref{eq:thm2_aux1} with \eqref{eq:thm2_aux2}, we see that if $c_0$ is chosen so that $cc'c_0^2>1$, then the upper bound
$R_q^q\rho_N^{1-q}+\rho_Ne^{-cN(\rho_N\land\, \rho_N^2)}$ in Theorem \ref{thm:mainresult} is dominated by $R_q^q\rho_N^{1-q}$ as $\lambda\to 0$. Therefore, for sufficiently small $\lambda$,
\begin{align*}
\mathbb{E} \| \widehat{\CC}_{\widehat{\rho}_N} - \CC \| \lesssim R_q^q\rho_N^{1-q} \le \|\CC\| \,c(q) \bigg(\frac{\log(\lambda^{-d})}{N}\bigg)^{\frac{1-q}{2}},
\end{align*}
where $c(q)$ is a constant that only depends on $q$.
\end{proof}

\section{Application in Ensemble Kalman Filters}\label{sec:ensembleKalman}

\begin{proof}[Proof of Theorem \ref{thm:LEKIExpectation}.]
First, we write
\begin{align}\label{eq:EnKF_aux_1}
    |\upost_n - \upost_n^{\star} |= |(\msK(\widehat{\CC})-\msK(\CC)) (y  - \mcA\upr_n- \eta_n)| \le \|\msK(\widehat{\CC})-\msK(\CC)\||y  - \mcA\upr_n- \eta_n |.
\end{align}
For the first term in \eqref{eq:EnKF_aux_1}, it follows by the continuity of the Kalman gain operator \cite[Lemma 4.1]{kwiatkowski2015convergence} that
\begin{align}\label{eq:EnKF_aux_2}
\|\msK(\widehat{\CC})-\msK(\CC)\|\le \|\widehat{\CC}-\CC\|\|\A\|\|\Gamma^{-1}\|\left(1+\|\CC\| \|\A\|^2\|\Gamma^{-1}\|\right).
\end{align}
Combining the inequalities \eqref{eq:EnKF_aux_1}, \eqref{eq:EnKF_aux_2}, and Theorem \ref{thm:smalllengthscale} gives that
\begin{align*}
\E \left[ | \upost_n - \upost_n^{\star} | \mid u_n,\eta_n\right]\lesssim \|\A\|\|\Gamma^{-1}\| |y-\A u_n-\eta_n |\,\E\|\widehat{\CC}-\CC\|\lesssim c \bigg(\sqrt{\frac{\lambda^{-d}}{N}} \lor \frac{\lambda^{-d}}{N}\bigg),
\end{align*}
where $c=\|\A\|\|\Gamma^{-1}\| \|\CC\| |y-\A u_n-\eta_n|$. Applying the same argument to the perturbed observation EnKF update with localization,  $\upost_n^{\,\rho}$, Theorem \ref{thm:smalllengthscale} gives that 
\begin{align*}
\E \left[ | \upost_n^{\,\rho}- \upost_n^{\star}| \mid u_n,\eta_n\right]\lesssim c\bigg[ c (q)\bigg(\frac{\log(\lambda^{-d})}{N}\bigg)^{\frac{1-q}{2}}\bigg],
\end{align*}
where $c=\|\A\|\|\Gamma^{-1}\| \|\CC\| |y-\A u_n-\eta_n|$ and $c(q)$ is a constant that depends only on $q$.
\end{proof}

\section{Conclusions, Discussion, and Future Directions}\label{sec:Conclusions}
This paper has studied thresholded estimation of sparse covariance operators, lifting the theory of sparse covariance matrix estimation from finite to infinite dimension. We have established non-asymptotic bounds on the estimation error in terms of the sparsity
level of the covariance and the expected supremum of the field. In the challenging regime where the correlation lengthscale is small, we have shown that estimation via thresholding achieves an exponential improvement in sample complexity over the standard sample covariance estimator. As an application of the theory, we have demonstrated the advantage of using thresholded covariance estimators within ensemble Kalman filters. While our focus has been on studying the statistical benefit of estimation via thresholding, sparsifying the covariance estimator can also lead to significant computational speed-up in downstream tasks \cite{furrer2006covariance,chen2023linear,chen2023scalable}.

As mentioned in the discussion of Theorem \ref{thm:smalllengthscale}, a natural question is whether the convergence rate of our thresholded estimator is minimax optimal. For $\ell_q$-sparse covariance matrix estimation, \cite{cai2012optimal} established the minimax optimality of thresholded estimators. Inspired by the correspondence between our error bound \eqref{eq:smalllengthscalewiththresholding} and their optimal rate, we conjecture that our thresholded estimator is also minimax optimal in the infinite-dimensional setting.

Another interesting future direction is to relax the assumption of stationarity in our analysis of the small lengthscale regime. In finite dimension, \cite{cai2011adaptive} proposed adaptive thresholding estimators for sparse covariance matrix estimation that account for variability across individual entries and designed a data-driven choice of the prefactor $c_0$ through cross-validation. Other interesting extensions include covariance operator estimation for heavy-tailed distributions \cite{abdalla2022covariance} and robust covariance operator estimation \cite{goes2020robust,diakonikolas2023algorithmic}. Finally, connections with the thriving topics of infinite-dimensional regression \cite{mollenhauer2022learning} and operator learning \cite{de2023convergence, jin2022minimax} will be explored in future work. 

\medskip

{\bf{Acknowledgments}}
The authors are grateful for the support of NSF DMS-2237628, DOE DE-SC0022232, and the BBVA Foundation.

\bibliographystyle{imsart-nameyear} 
\bibliography{references}       

\begin{appendix}
 \section{Proof of Lemma \ref{lemma:hat_rhoN_rhoN}}\label{app:A}

This appendix contains the proof of Lemma \ref{lemma:hat_rhoN_rhoN}. We will use the following auxiliary result, which can be found in \cite[Lemma 2.10.6]{talagrand2022upper}.
\begin{lemma}\label{lemma:concentration-suprema}
    Under Assumption \ref{assumption:main1} (i), it holds with probability at least $1-2e^{-t}$ that
    \[
    \left|\frac{1}{N}\sum_{n=1}^{N}\sup_{x\in D} u_n(x) - \E\left[\sup_{x\in D} u(x)\right]\right| \le 
    \sqrt{\frac{2t}{N}}. 
    \]
\end{lemma}
\begin{proof}
By Gaussian concentration, $\sup _{x \in D} u(x)$ is $\sup _{x\in D} \operatorname{Var}\left[u(x)\right]$-sub-Gaussian.
Since under Assumption \ref{assumption:main1} (i), $\sup _{x\in D} \operatorname{Var}\left[u(x)\right] = 1,$  a Chernoff bound argument gives the result.
\end{proof}


\begin{proof}[Proof of Lemma \ref{lemma:hat_rhoN_rhoN}]
We first prove (A). Without loss of generality, we assume $c_0=1$ in the definition of $\widehat{\rho}_N$ and $\rho_N$ in Theorem \ref{thm:mainresult}. Let $t>0$ and define $\mathcal{E}_t$ to be the event on which $\left|\frac{1}{N}\sum_{n=1}^{N}\sup_{x\in D} u_n(x) - \E\left[\sup_{x\in D} u(x)\right]\right| \le t$. It holds on $\mathcal{E}_t$ that
\begin{align*}
    \widehat{\rho}_N
    &\le \frac{1}{N}\lor \frac{\E[\sup_{x \in D} u(x)]+t}{\sqrt{N}}\lor \frac{(\E[\sup_{x \in D} u(x)]+t)^2}{N}\\
    &\le \frac{1}{N}\lor \frac{2\E[\sup_{x \in D} u(x)]}{\sqrt{N}}\lor \frac{2t}{\sqrt{N}}\lor \frac{4(\E[\sup_{x \in D} u(x)])^2}{N}\lor \frac{4t^2}{N}\\
   &\le 4\rho_N \lor \frac{2t}{\sqrt{N}}\lor \frac{4t^2}{N},
\end{align*}
and $\mathbb{P} \left[\widehat{\rho}_N \le 4\rho_N \lor \frac{2t}{\sqrt{N}}\lor \frac{4t^2}{N} \right] \ge \mathbb{P}\bigl[\mathcal{E}_t \bigr] \ge 1-2e^{-Nt^2/2}$ by Lemma~\ref{lemma:concentration-suprema}. It follows then that 
   \begin{align*}
       \E \big[\,\widehat{\rho}_N^{\,p} \big]&=
       p\int_0^{\infty} t^{p-1}\,\mathbb{P}\big[\,  \widehat{\rho}_N\ge t\big] dt=p\int_{0}^{4\,\rho_N}t^{p-1}\,\mathbb{P}\big[\,  \widehat{\rho}_N\ge t\big] dt+p\int_{4\,\rho_N}^{\infty}t^{p-1}\,\mathbb{P}\big[\,  \widehat{\rho}_N\ge t\big] dt\\
       &\le (4\rho_N)^p+2p\int_{4\,\rho_N}^{\infty} t^{p-1}\,e^{-\frac{N}{2}\min \{\frac{Nt^2}{4},\frac{Nt}{4}\}} dt\lesssim_p \rho_N^p+ \frac{1}{N^p}\lesssim_p \rho_N^p.
   \end{align*}

We next show (B).  To prove \eqref{eq:cor42_aux1}, we can assume $c_0=1$ without loss of generality. Notice that
\begin{align*}
    \mathbb{P} [\widehat{\rho}_N<t\rho_N]&=\mathbb{P}\left[\Big(\frac{1}{N}<t\rho_N\Big) \bigcap \Big(\frac{1}{\sqrt{N}} \Big(\frac{1}{N}\sum_{n=1}^{N}\sup_{x\in D} u_{n}(x)\Big)< t\rho_N\Big)  \bigcap  \Big(\frac{1}{N}\Big(\frac{1}{N}\sum_{n=1}^{N}\sup_{x\in D} u_{n}(x)\Big)^2 <t\rho_N\Big) \right]\\
   & \hspace{-0.52cm}  = 1-\mathbb{P}\left[\Big(\frac{1}{N}\ge t\rho_N\Big) \bigcup \Big(\frac{1}{\sqrt{N}} \Big(\frac{1}{N}\sum_{n=1}^{N}\sup_{x\in D} u_{n}(x)\Big)\ge t\rho_N\Big)  \bigcup  \Big(\frac{1}{N}\Big(\frac{1}{N}\sum_{n=1}^{N}\sup_{x\in D} u_{n}(x)\Big)^2 \ge t\rho_N\Big)\right].
\end{align*}
We consider three cases.

\emph{Case 1:} If $\E [\sup_{x \in D} u (x)]<\frac{1}{\sqrt{N}}$, then $\rho_N=\frac{1}{N}$ and
$\mathbb{P} [\widehat{\rho}_N<t\rho_N]\le 1-\mathbb{P}\Big[\frac{1}{N}\ge t\rho_N\Big]=0$.

\emph{Case 2:} If $\frac{1}{\sqrt{N}}\le \E[\sup_{x \in D} u(x)]\le \sqrt{N}$, then $\rho_N=\frac{1}{\sqrt{N}}\E[\sup_{x \in D} u(x)]$ and
\begin{align*}
\mathbb{P} [\widehat{\rho}_N<t\rho_N]&\le 1-\mathbb{P}\left[\frac{1}{\sqrt{N}} \Big(\frac{1}{N}\sum_{n=1}^{N}\sup_{x\in D} u_{n}(x)\Big)\ge t\rho_N\right]
=1-\mathbb{P}\left[\frac{1}{N}\sum_{n=1}^{N}\sup_{x\in D} u_{n}(x)\ge t\, \E [\sup_{x \in D} u(x)]\right]\\
&\le \mathbb{P}\left[\Big|\frac{1}{N}\sum_{n=1}^{N}\sup_{x\in D} u_n(x) - \E[\sup_{x \in D} u(x)]\Big|\ge (1-t)\E[\sup_{x \in D} u(x)]\right]\\
&\le 2\exp\Big(-\frac{1}{2}(1-t)^2 N(\E[\sup_{x \in D} u(x)])^2\Big),
\end{align*}
where the last step follows by Lemma \ref{lemma:concentration-suprema}. 

\emph{Case 3:} If $\E[\sup_{x \in D} u(x)]>\sqrt{N}$, then $\rho_N=\frac{1}{N}(\E[\sup_{x \in D} u(x)])^2$ and
\begin{align*}
    \mathbb{P} [\widehat{\rho}_N<t\rho_N]&\le 1-\mathbb{P}\left[\frac{1}{N}\Big(\frac{1}{N}\sum_{n=1}^{N}\sup_{x\in D} u_{n}(x)\Big)^2 \ge t\rho_N\right]
    =1- \mathbb{P}\left[\Big|\frac{1}{N}\sum_{n=1}^{N}\sup_{x\in D} u_{n}(x)\Big| \ge \sqrt{t} \,\E [\sup_{x \in D} u(x)]\right]\\
    &\le  \mathbb{P}\left[\Big|\frac{1}{N}\sum_{n=1}^{N}\sup_{x\in D} u_n(x) - \E[\sup_{x \in D} u(x)]\Big|\ge (1-\sqrt{t})\E[\sup_{x \in D} u(x)]\right]\\
&\le 2\exp\Big(-\frac{1}{2}(1-\sqrt{t})^2 N(\E[\sup_{x \in D} u(x)])^2\Big).
\end{align*}

Combining the three cases above and noticing that $(1-\sqrt{t})^2\le (1-t)^2$ for $t\in (0,1)$ yields the first inequality in \eqref{eq:cor42_aux1}. To prove \eqref{eq:cor42_aux2}, recall that $1\le c_0\le \sqrt{N}$ in the definition of $\rho_N$.
If $\E[\sup_{x \in D} u(x)]<1/\sqrt{N}$, then \eqref{eq:cor42_aux2} is trivial. If $\frac{1}{\sqrt{N}}\le \E[\sup_{x \in D} u(x)]\le \sqrt{N}$, then $\rho_N=\frac{c_0}{\sqrt{N}}\E [\sup_{x \in D} u(x)]$ and $N(\E [\sup_{x \in D} u(x)])^2 = \frac{N^2\rho_N^2}{c_0^2}\ge N\rho_N^2$, so that 
\[
2\,e^{-\frac{1}{2}(1-\sqrt{t})^2 N (\E [\sup_{x \in D} u(x)])^2}\mathbf{1}\big\{\E[\sup_{x \in D} u(x)]\ge 1/\sqrt{N}\big\}\le 2\,e^{-\frac{1}{2}(1-\sqrt{t})^2 N \rho_N^2}.
\]
If $\E[\sup_{x \in D} u(x)]>\sqrt{N}$, then $\rho_N=\frac{c_0}{N}(\E[\sup_{x \in D} u(x)])^2$ and $N(\E[\sup_{x \in D} u(x)])^2=\frac{N^2\rho_N}{c_0}\ge N^{3/2}\rho_N\ge N\rho_N$, so that 
\[
2\,e^{-\frac{1}{2}(1-\sqrt{t})^2 N (\E [\sup_{x \in D} u(x)])^2}\mathbf{1}\big\{\E[\sup_{x \in D} u(x)]\ge 1/\sqrt{N}\big\}\le 2\,e^{-\frac{1}{2}(1-\sqrt{t})^2 N \rho_N}.\qedhere
\]
\end{proof}

\section{Bound on Operator Norm}\label{app:B}

\begin{lemma}\label{lemma:operator_norm_bound}
   Let $D\subset \R^d$. For an integral operator $K$ on $L^2(D)$,
   \[
   (K\psi)(x):=\int_{D} k(x,x')\psi(x')dx', \quad \psi\in L^2(D),
   \]
   it holds that
   \[
   \|K\|^2 \le \left(\sup_{x}\int_{D}|k(x,x')|dx'\right)\cdot\left(\sup_{x'}\int_{D}|k(x,x')|dx\right).
   \]
   Further, if $k(x,x')=k(x',x)$, then
   \[
   \|K\| \le \sup_{x}\int_{D}|k(x,x')|dx'.
   \]
\end{lemma}
\begin{proof}
For any $\psi\in L^2(D)$ with $\|\psi\|_{L^2(D)}=1$,
\begin{align*}
\|K\psi\|_{L^2(D)}^2&=\int_{D}\Big(\int_D k(x,x')\psi(x') dx'\Big)^2 dx\\ &\le\int_{D}\Big(\int_D |k(x,x')|\cdot|\psi(x')| dx'\Big)^2 dx\\
&=\int_{D}\Big(\int_D \sqrt{|k(x,x')|}\cdot\sqrt{|k(x,x')|}|\psi(x')| dx'\Big)^2 dx\\
&\overset{\text{(i)}}{\le} \int_{D} \left(\int_{D}|k(x,x')|dx'\right) \left(\int_{D} |k(x,x')|\psi(x')^2 dx'\right) dx\\
&\le \left(\sup_{x}\int_{D}|k(x,x')|dx'\right)\cdot \int_{D}\int_{D} |k(x,x')|\psi(x')^2 dx'dx\\
&\le \left(\sup_{x}\int_{D}|k(x,x')|dx'\right)\cdot \int_{D}\left(\int_{D} |k(x,x')| dx\right)\psi(x')^2dx'\\
&\le\left(\sup_{x}\int_{D}|k(x,x')|dx'\right)\cdot\left(\sup_{x'}\int_{D}|k(x,x')|dx\right)\cdot \left(\int_{D}\psi(x')^2 dx'\right)\\
&=\left(\sup_{x}\int_{D}|k(x,x')|dx'\right)\cdot\left(\sup_{x'}\int_{D}|k(x,x')|dx\right),
\end{align*}
where (i) follows by Cauchy-Schwarz inequality. Therefore,
\[
\|K\|^2=\sup_{\|\psi\|_{L^2(D)}=1} \|K\psi\|_{L^2(D)}^2 \le \left(\sup_{x}\int_{D}|k(x,x')|dx'\right)\cdot\left(\sup_{x'}\int_{D}|k(x,x')|dx\right).
\]
Further, if $k(x,x')=k(x',x)$, then
\[
\|K\| \le \sqrt{\left(\sup_{x}\int_{D}|k(x,x')|dx'\right)\cdot\left(\sup_{x'}\int_{D}|k(x,x')|dx\right)}=\sup_{x}\int_{D}|k(x,x')|dx',
\]
which completes the proof.
\end{proof}

\section{Auxiliary Technical Result}\label{app:C}

\begin{lemma}\label{lemma:technique}
For any $\alpha>\beta>0$ and $q>0$, it holds that
\[
\int_{0}^{\alpha-\beta} e^{-\frac{1}{2}t^2}(\alpha-t)^{-q-1}\,dt\le \frac{2\beta^{-q}}{q}e^{-\frac{(\alpha-\beta)^2}{8}}+\frac{1}{q}\Big(\frac{\alpha}{2}\Big)^{-q}.
\]
\end{lemma}
\begin{proof}
Integrating by parts gives that
\begin{align*}
\int_{0}^{\alpha-\beta} e^{-\frac{1}{2}t^2}(\alpha-t)^{-q-1}dt&=\frac{\beta^{-q}}{q}e^{-\frac{(\alpha-\beta)^2}{2}}-\frac{\alpha^{-q}}{q}+\int_{0}^{\alpha-\beta} e^{-\frac{1}{2}t^2}t\frac{(\alpha-t)^{-q}}{q}dt\\
&=\frac{\beta^{-q}}{q}e^{-\frac{(\alpha-\beta)^2}{2}}-\frac{\alpha^{-q}}{q}+\left(\int_{0}^{\frac{\alpha-\beta}{2}}+\int_{\frac{\alpha-\beta}{2}}^{\alpha-\beta}\right) e^{-\frac{1}{2}t^2}t\frac{(\alpha-t)^{-q}}{q}dt.
\end{align*}
First,
\[
\int_{0}^{\frac{\alpha-\beta}{2}} e^{-\frac{1}{2}t^2}t\frac{(\alpha-t)^{-q}}{q}dt \le \frac{1}{q}\Big(\alpha-\frac{\alpha-\beta}{2}\Big)^{-q}\int_0^{\frac{\alpha-\beta}{2}} e^{-\frac{1}{2}t^2} t \,dt\le \frac{1}{q}\Big(\frac{\alpha+\beta}{2}\Big)^{-q} \le \frac{1}{q}\Big(\frac{\alpha}{2}\Big)^{-q}.
\]
Second,
\[
\int_{\frac{\alpha-\beta}{2}}^{\alpha-\beta} e^{-\frac{1}{2}t^2}t\frac{(\alpha-t)^{-q}}{q}dt\le \frac{1}{q}(\alpha-(\alpha-\beta))^{-q}\int_{\frac{\alpha-\beta}{2}}^{\alpha-\beta} e^{-\frac{1}{2}t^2}t \,dt \le \frac{\beta^{-q}}{q} e^{-\frac{(\alpha-\beta)^2}{8}}.
\]
Thus,
\begin{align*}
\int_{0}^{\alpha-\beta} e^{-\frac{1}{2}t^2}(\alpha-t)^{-q-1}dt&\le \frac{\beta^{-q}}{q}e^{-\frac{(\alpha-\beta)^2}{2}}-\frac{\alpha^{-q}}{q}+\frac{\beta^{-q}}{q}e^{-\frac{(\alpha-\beta)^2}{8}}+\frac{1}{q}\Big(\frac{\alpha}{2}\Big)^{-q}\\
&\le \frac{2\beta^{-q}}{q}e^{-\frac{(\alpha-\beta)^2}{8}}+\frac{1}{q}\Big(\frac{\alpha}{2}\Big)^{-q}. \qedhere
\end{align*}
\end{proof}

 \end{appendix}

\end{document}